\documentclass[12pt]{amsart}
 \usepackage[dvips]{graphicx}

\usepackage{amsmath, amssymb}
\usepackage{pifont}

  \numberwithin{equation}{section}
  \newtheorem{theorem}{Theorem}[section]  
  \newtheorem{theorem?}{``Theorem''}[section]  
  
  \newtheorem{proposition}[theorem]{Proposition}
  \newtheorem{lemma}[theorem]{Lemma}

\theoremstyle{definition}
\theoremstyle{remark}
  \newtheorem{remark}[theorem]{Remark}  
\newcommand{\R}{{\mathbb R}}
\newcommand{\C}{{\mathbb C}}
\newcommand{\Q}{{\mathbb Q}}
\newcommand{\N}{{\mathbb N}}
\newcommand{\Z}{{\mathbb Z}}

\renewcommand{\a}{\alpha}

\renewcommand{\d}{\partial}

\newcommand{\card}{{\Pisymbol{psy}{"23}}}

\begin{document}
\title[Oscillatory integrals]
{
Toric resolution of singularities in 
a certain class of $C^{\infty}$ functions 
and 
asymptotic analysis 
\\
of oscillatory integrals
} 
\author{Joe Kamimoto and Toshihiro Nose}
\address{Faculty of Mathematics, Kyushu University, 
Motooka 744, Nishi-ku, Fukuoka, 819-0395, Japan} 
\email{
joe@math.kyushu-u.ac.jp}
\email{ 
t-nose@math.kyushu-u.ac.jp}
\thanks{
The first author was supported by 
Grant-in-Aid for Scientific Research (C) (No. 22540199), 
Japan Society for the Promotion of Science. 
}
\maketitle



\begin{abstract}
In a seminal work of A. N. Varchenko, 
the behavior at infinity 
of oscillatory integrals with real analytic phase 
is precisely investigated by 
using the theory of toric varieties 
based on the geometry of the Newton polyhedron
of the phase. 
The purpose of this paper is to generalize
his results to the case that
the phase is contained in 
a certain class of $C^{\infty}$ functions. 
The key in our analysis is 
a toric resolution of singularities in 
the above class of $C^{\infty}$ functions.
The properties of poles of local zeta functions, 
which are closely related
to the behavior of oscillatory integrals, 
are also studied under the associated situation. 
\end{abstract}


\tableofcontents


\section{Introduction}
In this paper, we investigate the asymptotic behavior of 
oscillatory integrals, that is, integrals of the form
\begin{equation}
I(t;\varphi)=
\int_{\R^n}e^{it f(x)}\varphi(x)dx,
\label{eqn:1.1}
\end{equation}
for large values of the real parameter $t$, 
where 
$f$ is a real-valued $C^{\infty}$ smooth function 
defined on $\R^n$ and     
$\varphi$ is a complex-valued $C^{\infty}$ smooth function 
whose support is contained  
in a small neighborhood of the origin in $\R^n$.
Here $f$ and $\varphi$ are 
called the {\it phase} and the {\it amplitude}, 
respectively. 

By the principle of stationary phase, 
the main contribution in the
behavior of the integral (\ref{eqn:1.1}) 
as $t\to+\infty$ is given by the local properties 
of the phase on neighborhoods of its critical points.
When the phase has a nondegenerate critical point, 
i.e., the $n\times n$ matrix $\nabla^2 f$ 
is invertible, 
the Morse lemma implies that
there exists a coordinate where 
$f$ is locally expressed as 
$x_1^2+\cdots+x_k^2-x_{k+1}^2-\cdots-x_n^2$ 
with some $k$, 
This fact easily gives the asymptotic expansion of $I(t;\varphi)$
through the computation of Fresnel integrals.
On the other hand,  
the situation at degenerate critical points is quite different.
There are very few cases that direct computations are available
for the analysis of $I(t;\varphi)$ by 
using a smooth change of coordinates only. 
Up to now, 
there have been many studies about the degenerate case, 
which develop more intrinsic and ingenious methods
to see the behavior of $I(t;\varphi)$ 
(see \cite{var76},\cite{sch91},\cite{ste93},
\cite{pss99},\cite{gre04},\cite{py04},
\cite{dns05},\cite{gre08},\cite{gre09},\cite{gre10ma},
\cite{gre10jam},\cite{ikm10},\cite{im11tams}
\cite{im11jfaa},\cite{cgp12},
\cite{ckn12},\cite{nos12}, etc.).  
Analogous studies about oscillatory integral operators 
are seen in 
\cite{ste93},\cite{ps97},\cite{ryc01},\cite{gs02},\cite{gre05},\cite{gpt07}, 
etc. 

The following classical results need the hypothesis 
of the real analyticity of the phase.
By using Hironaka's resolution of singularities, 
it is known (c.f. \cite{jea70},\cite{mal74}) that 
$I(t;\varphi)$ admits an asymptotic expansion 
(see (\ref{eqn:3.1}) in Section~3). 
More precisely, Varchenko \cite{var76} investigates 
the leading term of this asymptotic expansion
by using the theory of toric varieties 
based on the geometry of 
the Newton polyhedron of the phase  
under a certain nondegeneracy condition 
on the phase
(see Theorem~3.1 in Section~3).  
Since his study, 
the investigation of the behavior of 
oscillatory integrals has been more closely linked with 
the theory of singularities. 
Refer to the excellent exposition 
\cite{agv88} for studies in this direction. 
The investigation under the nondegeneracy
hypothesis has been 
developed in 
\cite{ds89},\cite{ds92},\cite{dls97},\cite{dns05}.

In the same paper \cite{var76}, 
Varchenko investigates the  
two-dimensional case in more detail. 
When the phase is real analytic, 
he proves the existence of a good coordinate
system, 
which is called an {\it adapted coordinate},
and gives analogous results 
without the nondegeneracy hypothesis
by using this coordinate. 
His proof is based on a two-dimensional resolution
of singularities result. 
Notice that the adapted coordinate
may not exist in dimensions higher than two. 
Later, his two-dimensional results have been improved in 
\cite{pss99},\cite{gre04},\cite{py04},\cite{gre09},\cite{ikm10},\cite{im11tams},\cite{im11jfaa},
which are inspired by the work of
Phong and Stein in their seminal paper \cite{ps97}. 

In higher dimensions, 
recent interesting studies 
\cite{gre08},\cite{gre10ma},\cite{gre10jam},\cite{cgp12} 
also emphasis
the importance of the relationship
between  
behavior of 
oscillatory integrals 
and  
resolution of singularities 
for the phase. 
Observing these studies and the two-dimensional 
works mentioned above, we see that 
explicit and elementary approaches 
to the resolution of singularities 
are useful for quantitative investigation of
the decay rate of oscillatory integrals.

In this paper, 
we generalize the above results of 
Varchenko \cite{var76},
under the same nondegeneracy hypothesis,  
to the case that
the phase is contained in a certain class of 
$C^{\infty}$ functions 
including real analytic functions.
This class is denoted by 
$\hat{\mathcal E}(U)$, 
where 
$U$ is an open neighborhood of the origin in $\R^n$. 
Under the nondegeneracy condition,  
we construct a toric resolution of singularities
in the class $\hat{\mathcal E}(U)$.  
Using this resolution, 
we show that 
$I(t;\varphi)$ has 
an asymptotic expansion of the same form 
as in the real analytic phase case 
and succeed to
generalize the above results of Varchenko.
Moreover, we give an explicit formula of the coefficient
of the leading term of the above asymptotic expansion. 

Let us explain the properties of 
the class $\hat{\mathcal E}(U)$ 
in more detail. 
In the above earlier many investigations, 
the function $\gamma$-{\it part}, 
which corresponds to each face $\gamma$ 
of the Newton polyhedron of the phase, 
plays an important role.
By using summation,
the $\gamma$-part is simply defined as a function for 
every face $\gamma$ in the real analytic case.
From the viewpoint of this definition,    
the $\gamma$-part is considered as a formal power series 
when $\gamma$ is noncompact and the phase is only smooth.   
This $\gamma$-part may not become a function, so
it is not useful for our analysis. 
From convex geometrical points of view 
(c.f. \cite{zie95}), we give another definition
of the $\gamma$-part, which always becomes 
a function defined near the origin
(see Section~2.3).
This definition is a natural generalization
of that in the real analytic case.
We remark that not all smooth functions admit 
the $\gamma$-part for every face $\gamma$ of
their Newton polyhedra in our sense. 
The class $\hat{\mathcal E}(U)$ is defined to be the set of 
$C^{\infty}$ functions admitting the $\gamma$-part
for every face $\gamma$ of its Newton polyhedron  
(see Section~2.4).
Many kinds of $C^{\infty}$ functions are contained 
in this class. 
In particular, it contains 
the Denjoy-Carleman quasianalytic classes,
which are
interesting classes of $C^{\infty}$ functions and
have been studied 
from various points of view
(c.f. \cite{bm04},\cite{thi08}). 
The most important property of the class 
$\hat{\mathcal E}(U)$ is that 
its element is    
generated by finite monomials whose powers
are contained in its Newton polyhedron.
This property plays a crucial role in the 
construction of a toric resolution of singularities
in the class 
$\hat{\mathcal E}(U)$.

There have been many attempts to understand 
the behavior of oscillatory integrals 
with {\it smooth} phases.  
Explicit asymptotic expansions of $I(t;\varphi)$ 
are computed 
in the case of one-dimensional nonflat phases  
(see the monograph \cite{ste93}) 
and 
in the case of finite line type convex phases
(see \cite{sch91}). 
In the two-dimensional case, 
strong results are also obtained by using 
an adapted coordinate, 
which exists even in the smooth case,
in \cite{gre09},\cite{ikm10},\cite{im11tams},\cite{im11jfaa}.   
(As for analogous studies about oscillatory integral operators, 
the one-dimensional case  
has been completely understood 
when the phase is nonflat in \cite{ryc01},\cite{gre05}.)  
On the other hand, 
a simple example given by Iosevich and Sawyer \cite{is97} 
in two dimensions
shows that  
some kind of restriction 
like the class $\hat{\mathcal E}(U)$ 
is necessary to generalize the results of Varchenko 
directly (see Section~11.4). 
The smooth case is difficult to deal with 
because analytical information 
of functions around critical points does not always appear in 
the geometry of their Newton polyhedra.   


It is known  
(see, for instance, 
\cite{igu78}, \cite{agv88} and Section~10.1 in this paper) 
that 
the asymptotic analysis of the oscillatory integral (\ref{eqn:1.1}) 
can be reduced to an investigation of the poles of the functions 
$Z_{\pm}(s;\varphi)$ (see (\ref{eqn:5.1}) in Section~9), 
which are similar to the local zeta function 
\begin{equation*}
Z(s;\varphi)=
\int_{\R^n} |f(x)|^s\varphi(x)dx,
\end{equation*}
where $f$, $\varphi$ are  
as in (\ref{eqn:1.1}) and $f$ vanishes at a critical point. 
The substantial analysis in this paper is to  
investigate the properties of poles of 
the local zeta function $Z(s;\varphi)$ and 
the above functions $Z_{\pm}(s;\varphi)$ by using 
the geometrical properties of the Newton polyhedron 
of the function $f$.
We also give new results relating to  
the poles of these functions. 

This paper is organized as follows. 
In Section~2, 
after explaining  
some important notions in convex geometry, 
we give the definition of  
Newton polyhedra and explain related important words
in our analysis. 
Moreover, after generalizing the concept of the $\gamma$-part,
we introduce the classes 
$\hat{\mathcal E}[P](U)$ and 
$\hat{\mathcal E}(U)$ of $C^{\infty}$ functions. 
In Section~3, 
we state main results relating to oscillatory integrals. 
Some parts of the results are new even when the phase is 
real analytic. 
In Section~4, 
we consider elementary convex geometrical properties of 
polyhedra, which are useful in this paper. 
In Section~5, 
basic properties of generalized $\gamma$-part are 
investigated. 
In Section~6, 
more detailed properties of the class 
$\hat{\mathcal E}(U)$ are investigated, 
which play important roles in the resolution
of singularities. 
In Section~7,
we overview the method to construct toric varieties
from a given polyhedron.
In Section~8,   
we construct a resolution of singularities 
in the class $\hat{\mathcal E}(U)$
under the nondegeneracy condition in \cite{var76}.  
In Section~9,
we investigate the properties of 
poles of the local zeta function 
$Z(s;\varphi)$ and 
the functions $Z_{\pm}(s;\varphi)$ by using 
the resolution of singularities constructed 
in Section~8.
In Section~10, 
we give proofs of theorems on the behavior
of oscillatory integrals stated 
in Section~3. 
Furthermore, we give explicit formulae for 
the leading term of the asymptotic expansion 
of $I(t;\varphi)$.
In Section~11, 
we give concrete computations for some examples, 
which are not directly covered in earlier investigations. 

{\it Notation and symbols.}\quad
\begin{enumerate}
\item 
We denote by $\Z_+, \Q_+, \R_+$ the subsets consisting of 
all nonnegative numbers in $\Z,\Q,\R$, respectively.
For $s\in\C$, $\Re(s)$ expresses the real part of $s$. 
\item
We use the multi-index as follows.
For $x=(x_1,\ldots,x_n),y=(y_1,\ldots,y_n) \in\R^n$, 
$\a=(\a_1,\ldots,\a_n)\in\Z_+^n$, 
define
\begin{eqnarray*}
&& 
|x|=\sqrt{|x_1|^{2}+\cdots +|x_n|^{2}}, \quad 
\langle x,y\rangle =x_1 y_1+\dots+x_n y_n, 
\\
&& 
x^{\a}=x_1^{\a_1}\cdots x_n^{\a_n}, \quad
\partial^{\a}=
\left(\frac{\partial}{\partial x_1}\right)^{\a_1}\cdots
\left(\frac{\partial}{\partial x_n}\right)^{\a_n},  
\\
&&
\langle \a\rangle =\a_1+\cdots+\a_n, \quad
\a!=\a_1!\cdots \a_n! \quad (0!=1!=1). 
\end{eqnarray*}
\item
For $A,B\subset \R^n$ and $c\in\R$, 
we set 
$$
A+B=\{a+b\in\R^n; a\in A \mbox{ and } b\in B\},\quad
c\cdot A=\{ca\in\R^n; a\in A\}.
$$
\item
For a finite set $A$, 
$\card A$ means the cardinality of $A$. 
\item
For a nonnegative real number $r$ and 
a subset $I$ in $\{1,\ldots,n\}$, the map 
$T_I^{r}:\R^n\to\R^n$ is defined by 
\begin{equation}
(z_1,\ldots,z_n)=T_I^{r}(x_1,\ldots,x_n)\,\, 
\mbox{ with }\,\, z_j:=\begin{cases}
r& 
\quad \mbox{for $j\in I$}, \\
x_j&
\quad \mbox{otherwise}.
\end{cases}
\label{eqn:1.2}
\end{equation}
We define $T_I:=T_I^0$.
For a set $A$ in $\R^n$, 
the image of $A$ by $T_I$ is denoted by $T_I(A)$.  
When $A=\R^n$ or $\Z_+^n$, 
its image is expressed as 
\begin{equation}
T_I(A)=\{x\in A; x_j=0 \mbox{ for $j\in I$}\}.
\label{eqn:1.3}
\end{equation}
\item
For a $C^{\infty}$ function $f$, 
we denote by Supp($f$) the support of $f$, i.e., 
Supp($f$)$=\overline{\{x\in \R^n; f(x)\neq 0\}}$. 
\item For $x\in\R$, $\alpha>0$, 
the value of $e^{-1/|x|^{\alpha}}$ at the origin
is defined by $0$. 
Then $e^{-1/|x|^{\alpha}}$ is a $C^{\infty}$ function
defined on $\R$.
\end{enumerate}


\section{Newton polyhedra and 
the classes $\hat{\mathcal E}[P](U)$ and 
$\hat{\mathcal E}(U)$}

\subsection{Polyhedra}
Let us explain fundamental notions 
in the theory of convex polyhedra, 
which are necessary for our study.
Refer to \cite{zie95}   
for general theory of convex polyhedra.  

For $(a,l)\in \R^n\times\R$, 
let $H(a,l)$ and $H^+(a,l)$ be  
a hyperplane and 
a closed half space in $\R^n$ 
defined by
\begin{equation}
\begin{split}
&H(a,l):=\{x\in\R^n;\langle a,x\rangle =l\},\\
&H^+(a,l):=\{x\in\R^n;\langle a,x\rangle \geq l\},
\end{split}
\label{eqn:2.1}
\end{equation} 
respectively. 
A ({\it convex rational}) {\it polyhedron} is  
an intersection of closed halfspaces:
a set $P\subset\R^n$ presented in the form
$
P=\bigcap_{j=1}^N H^+(a^j,l_j)
$
for some $a^1,\ldots,a^N\in\Z^n$ and 
$l_1,\ldots,l_N \in\Z$.

Let $P$ be a polyhedron in $\R^n$. 
A pair $(a,l)\in \Z^n\times\Z$ is said to be 
{\it valid} for $P$ 
if $P$ is contained in $H^+(a,l)$.
A {\it face} of $P$ is any set of the form 
$
F=P\cap H(a,l),
$
where $(a,l)$ is valid for $P$. 
Since $(0,0)$ is always valid, 
we consider $P$ itself as a trivial face of $P$;
the other faces are called {\it proper faces}.  
Conversely, 
it is easy to see that any face is a polyhedron. 
Considering the valid pair $(0,-1)$, 
we see that the empty set is always a face of $P$. 
Indeed, $H^+(0,-1)=\R^n$, but $H(0,-1)=\emptyset$.
The {\it dimension} of a face $F$ is the dimension of 
its affine hull of $F$ 
(i.e., the intersection of all affine flats that 
contain $F$), which is denoted by $\dim(F)$. 
The faces of dimensions $0,1$ and $\dim(P)-1$
are called {\it vertices}, {\it edges} and 
{\it facets}, respectively. 
The {\it boundary} of a polyhedron $P$, denoted by 
$\d P$,  
is the union of all proper faces of $P$.  
For a face $F$, $\d F$ is similarly defined. 

\subsection{Newton polyhedra}

Let $f$ 
be a real-valued $C^{\infty}$ function defined 
on an open neighborhood of the origin in $\R^n$. 
Denote by $\hat{f}(x)$
the Taylor series of $f$ at the origin, i.e.,   
$$
\hat{f}(x)=\sum_{\alpha\in{\Z}_+^n} c_{\alpha}x^{\alpha} 
\quad \mbox{ with $c_{\alpha}=
\dfrac{\partial^{\alpha} f(0)}{\alpha!}$}.
$$ 
The {\it Newton polyhedron} of $f$
is the integral polyhedron: 
$$
\Gamma_+(f)=
\mbox{the convex hull of the set 
$\bigcup \{\a+\R_+^n;c_{\a}\neq 0\}$ in $\R_+^n$}
$$
(i.e., the intersection 
of all convex sets 
which contain $\bigcup \{\a+\R_+^n;\a\in S_f\}$). 
It is known (cf. \cite{zie95}) that the Newton polyhedron 
$\Gamma_+(f)$ is a polyhedron. 
The union of the compact faces of 
the Newton polyhedron $\Gamma_+(f)$ is called 
the {\it Newton diagram} $\Gamma(f)$ of $f$, 
while the boundary of $\Gamma_+(f)$ 
is denoted by 
$\d\Gamma_+(f)$. 
The {\it principal part} of $f$ is defined by 
$f_*(x)=\sum_{\alpha\in\Gamma(f)\cap\Z_+^n}
c_{\alpha}x^{\alpha}.$
Note that $\Gamma_+(f)=\Gamma_+(f_*)$.

A $C^{\infty}$ function $f$ is said to be {\it convenient} if 
the Newton polyhedron $\Gamma_+(f)$ 
intersects all the coordinate axes. 

%
We assume that $f$ is {\it nonflat}, i.e., 
$\Gamma_+(f)\neq \emptyset$.  
Let $q_*$ be the point at which 
the line $\a_1=\cdots=\a_n$ in $\R^n$ 
intersects the boundary of $\Gamma_+(f)$. 
The coordinate of $q_*$ is called 
the {\it Newton distance} of $\Gamma_+(f)$, 
which is denoted by $d(f)$, i.e., 
$q_*=(d(f),\ldots,d(f))$. 
The face whose relative interior contains $q_*$ is 
called the {\it principal face} of
$\Gamma_+(f)$, which is denoted by $\tau_*$. 
The codimension of $\tau_*$ is called the 
{\it Newton multiplicity} of $\Gamma_+(f)$, 
which is denoted by $m(f)$.   
Here, when $q_*$ is a vertex of $\Gamma_+(f)$, 
$\tau_*$ is the point $q_*$ and $m(f)=n$.
\subsection{
The $\gamma$-part
}

Let $f$ be a real-valued $C^{\infty}$ function on 
an open neighborhood $V$ of the origin in $\R^n$,
$P\subset \R_+^n$ a nonempty polyhedron 
containing $\Gamma_+(f)$
and  
$\gamma$ a face of $P$. 
Note that this polyhedron $P$ satisfies the condition
$P+\R_+^n\subset P$, 
which will be shown in Lemma~4.1, below.  
We say that  
$f$ {\it admits the $\gamma$-part} 
on an open neighborhood $U\subset V$ of the origin 
if for any $x\in U$ the limit:
\begin{equation}
\lim_{t\to 0}\frac{f(t^{a_1}x_1,\ldots,t^{a_n}x_n)}{t^l}
\label{eqn:2.2}
\end{equation}
exists for {\it all} valid pairs 
$(a,l)=(a_1,\ldots,a_n,l)\in\Z_+^n\times\Z_+$
defining $\gamma$
(i.e., $H(a,l)\cap P=\gamma$). 
Proposition~5.2 (iii), below, implies that when 
$f$ admits the $\gamma$-part, 
the above limits take the same value for 
any $(a,l)$, which is denoted by $f_{\gamma}(x)$. 
We consider $f_{\gamma}$ as 
the function on $U$, which  
is called the $\gamma$-part of $f$ on $U$.  
\begin{remark}
We give many remarks on the $\gamma$-part. 
Some of them are not trivial and
they will be shown later. 
\begin{enumerate}
\item 
The readers might feel that
``all'' is too strict
in the above definition of the admission of the 
$\gamma$-part. 
Actually, even if ``all'' is replaced by ``some'' 
in the definition, 
this exchange does not affect the analysis in this paper. 
This subtle issue will be discussed in 
Section~6.4.
\item 
If $f$ admits the $\gamma$-part $f_{\gamma}$
on $U$, then
$f_{\gamma}$ has the quasihomogeneous property:
\begin{equation}
f_{\gamma}(t^{a_1} x_1,\ldots, t^{a_n} x_n)=t^l f_{\gamma}(x) 
\mbox{\,\, for \,\, 
$t\in(0,1]$ and $x\in U$},
\label{eqn:2.3}
\end{equation}
where $(a,l)$ is a valid pair defining $\gamma$ 
(see Lemma~5.4 (i)).
\item 
The above $\gamma$-part $f_{\gamma}$  
is a $C^{\infty}$ function defined on $U$ 
(see Proposition~5.2 (iv)). 
Moreover, the above quasihomogeneity (\ref{eqn:2.3})
implies that 
$f_{\gamma}$ can be uniquely extended to 
a $C^{\infty}$ function with the property 
(\ref{eqn:2.3}) defined on 
much wider regions (see Lemma~5.4 (ii)). 
This function is also denoted by $f_{\gamma}$. 
\item 
When $\gamma=P$, 
$f$ always admits the 
$\gamma$-part on $V$ and $f_P=f$. 
In fact, consider the case when $(a,l)=(0,0)$. 
\item
   When $\gamma$ is contained in 
   some coordinate plane,
   $f$ always admits the $\gamma$-part on $V$.
   Indeed, for any pair $(a,l)$ defining $\gamma$,
   we have $l=0$ and
   the limit (\ref{eqn:2.2}) always exists.
\item
For a compact face $\gamma$ of $\Gamma_+(f)$, 
$f$ always admits the $\gamma$-part near the origin
and $f_{\gamma}(x)$ equals the polynomial 
$\sum_{\a\in\gamma\cap\Z_+^n}
c_{\alpha}x^{\alpha}$, which coincides with the 
definition of  
the $\gamma$-part of $f$ in 
\cite{var76},\cite{agv88} 
(see Proposition~5.2 (iii)). 
\item
Let $\gamma$ be a noncompact face of $\Gamma_+(f)$.
If $f$ admits the $\gamma$-part $f_{\gamma}$ on $U$, 
then the Taylor series at the origin of $f_{\gamma}$ is 
$\sum_{\a\in\gamma\cap\Z_+^n}
c_{\alpha}x^{\alpha}$
(see Lemma~5.3).
\item 
If $f$ is real analytic on $V$ and 
$\gamma$ is a face of $\Gamma_+(f)$, then 
$f$ admits the $\gamma$-part $f_{\gamma}$ 
on $U$ and, moreover,  
$f_{\gamma}$ is equal to a convergent power series
$\sum_{\a\in\gamma\cap\Z_+^n}
c_{\alpha}x^{\alpha}$ on $U$ 
(see Lemma~5.3).
\item
An example, 
which shows the case of non-admission of the $\gamma$-part, 
will be given in Section~2.5.
This example also indicates that 
for a face $\gamma$ of $P$, 
the condition
$\gamma\cap\Gamma_+(f)=\emptyset$ does not always
mean that $f$ admits the $\gamma$-part:
$f_{\gamma}\equiv 0$. 
\end{enumerate}
\end{remark}

\subsection{
The classes 
$\hat{\mathcal E}[P](U)$ and 
$\hat{\mathcal E}(U)$}

Let 
$P\subset \R_+^n$ be a polyhedron (possibly an empty set)
satisfying $P+\R_+^n\subset P$
if $P\neq\emptyset$
and 
$U$ an open neighborhood of the origin. 
Denote by 
${\mathcal E}[P](U)$
the set of $C^{\infty}$ functions defined on $U$ 
whose Newton polyhedra are contained in $P$.
Moreover, when $P\neq \emptyset$, we denote by
$
\hat{{\mathcal E}}[P](U)
$
the set of the elements $f$
of ${\mathcal E}[P](U)$
admitting the $\gamma$-part on $U$ 
for any face $\gamma$ of $P$. 
We set $\hat{{\mathcal E}}[\emptyset](U)=\{0\}$,
i.e., 
the set consisting of only the function identically 
equaling zero on $U$. 
We define
$$
\hat{\mathcal E}(U)=
\{f\in C^{\infty}(U);f\in\hat{\mathcal E}[\Gamma_+(f)](U)\}.
$$ 
\begin{remark}
In the definition of $\hat{\mathcal E}[P](U)$, 
``any face'' 
can be replaced by ``any noncompact facet'' 
(see Section~6.4).
\end{remark}

\begin{remark}
The class $\hat{\mathcal E}(U)$ contains many kinds of 
$C^{\infty}$ functions. 
Here $U$ is a small open neighborhood of the origin in $\R^n$. 
\begin{enumerate}
\item
The function identically equaling zero
on $U$ is contained in $\hat{{\mathcal E}}(U)$. 
This easily follows from the definition.
\item
Every real analytic function defined on $U$
belongs to $\hat{\mathcal E}(U)$. 
This follows from Remark~2.1 (viii).
\item 
Every convenient $C^{\infty}$ function 
defined on $U$
belongs to $\hat{\mathcal E}(U)$.
This follows from Remark~2.1 (v), (vi).
\item
In the one-dimensional case, 
every nonflat $C^{\infty}$ function defined on $U$
belongs to $\hat{\mathcal E}(U)$. 
This is a particular case of the above (iii).
\item
The Denjoy-Carleman classes ${\mathcal E}_M(U)$
are contained in $\hat{\mathcal E}(U)$.
See Proposition~6.10. 
\end{enumerate}
\end{remark}
\begin{remark}
Tougeron \cite{tou68} shows that 
if a $C^{\infty}$ function $f$ has a critical point of 
``finite multiplicity'' (see \cite{agv85}, p121), 
then $f$ can be expressed as a
polynomial around the critical point 
by using smooth coordinate changes.
But, there are many elements in 
$\hat{\mathcal E}(U)$ or $\hat{\mathcal E}[P](U)$ 
do not satisfy this hypothesis. 
(Our classes contain all real analytic functions.)
\end{remark}
\begin{remark}
The classes $\hat{\mathcal E}[P](U)$ and 
$\hat{\mathcal E}(U)$ are useful for the investigation
of the behavior of
{\it weighted} oscillatory integrals:
$$
\tilde{I}(t;\varphi)
=\int_{\R^n} e^{it f(x)}g(x)\varphi(x)dx,
$$
where $f,\varphi$ are the same as in (\ref{eqn:1.1})
and $g$ is a {\it weight} function satisfying some conditions
(see \cite{knhir12},\cite{nos12}).
\end{remark}

More detailed properties of the classes 
$\hat{\mathcal E}[P](U)$ and 
$\hat{\mathcal E}(U)$ are investigated in 
Section~6 below.

\subsection{An example}
Let us consider the following two-dimensional example.
\begin{equation}
\begin{split}
&f_k(x)=f_k(x_1,x_2)=x_1^2x_2^2+x_1^ke^{-1/x_2^2},
\quad 
k\in\Z_+; \\
&P=
\{(\alpha_1,\alpha_2)\in\R_+^2;
\alpha_1\geq 1, \,\, \alpha_2\geq 1\}.
\end{split}
\end{equation}
Of course, $f_k$ is not real analytic around the origin. 
The set of the proper faces of $\Gamma_+(f_k)$ and $P$ consists of  
$\gamma_1,\gamma_2,\gamma_3$ and $\tau_1,\tau_2,\tau_3$, where
\begin{equation}
\begin{split}
&\gamma_1=\{(2,\a_2);\a_2\geq 2\},
\,\,
\gamma_2=\{(2,2)\},\,\,
\gamma_3=\{(\a_1,2);\a_1\geq 2\},\\
&
\tau_1=\{(1,\a_2);\a_2\geq 1\}, \,\,
\tau_2=\{(1,1)\},\,\,
\tau_3=\{(\a_1,1);\a_1\geq 1\}.
\end{split}
\end{equation}
It is easy to see that if $j=2,3$, 
then $f_k$ admits the $\gamma_j$-part
and $\tau_j$-part near the origin for all $k\in\Z_+$ and 
they are written as 
$
(f_k)_{\gamma_2}(x)=(f_k)_{\gamma_3}(x)=x_1^2x_2^2$
and 
$ 
(f_k)_{\tau_2}(x)=(f_k)_{\tau_3}(x)\equiv 0.
$
Consider the $\gamma_1$-part and $\tau_1$-part of $f_k$ for
$k\in\Z_+$.
The situation depends on the parameter $k$ as follows.
\begin{itemize}
\item $(f_0)_{\gamma_1}$ and $(f_0)_{\tau_1}$ cannot be defined.
\item $(f_1)_{\gamma_1}$ cannot be defined but 
$(f_1)_{\tau_1}(x)=x_1e^{-1/x_2^2}$.
\item $(f_2)_{\gamma_1}(x)=f(x)$ and $(f_2)_{\tau_1}(x)\equiv 0$.
\item If $k\geq 3$, then 
$f_{\gamma_1}(x)=x_1^2x_2^2$ and $f_{\tau_1}(x)\equiv 0$.
\end{itemize}
From the above, we see that
$f_k\in\hat{{\mathcal E}}(U)$ if and only if $k\geq 2$; 
$f_k\in\hat{{\mathcal E}}[P](U)$ if and only if $k\geq 1$.
Notice that  
$\tau_1\cap \Gamma_+(f_1)=\emptyset$ but
$(f_1)_{\tau_1}(x)=x_1e^{-1/x_2^2}\not\equiv 0$ 
(see Remark~2.1 (ix)).


\section{Main results}

Let us explain our results relating to the behavior 
of the oscillatory integral $I(t;\varphi)$ 
in (\ref{eqn:1.1}) as $t \to+\infty$. 

Throughout this section, 
the functions $f$, $\varphi$ satisfy 
the following conditions. 
Let $U$ be an open neighborhood of the origin in $\R^n$.
\begin{enumerate}
\item[(A)] 
$f:U \to \R$ is a $C^{\infty}$ smooth function satisfying 
that $f(0)=0$, $\nabla f(0)=(0,\ldots,0)$ and 
$\Gamma_+(f)\neq\emptyset$;
\item[(B)]
$\varphi:\R^n\to\C$ is a $C^{\infty}$ smooth function 
whose support is contained in $U$. 
\end{enumerate}

\subsection{Known results}

As mentioned in the Introduction,
by using Hironaka's resolution of singularities 
\cite{hir64}, an asymptotic expansion for 
$I(t;\varphi)$ 
is obtained (c.f. \cite{jea70},\cite{mal74}). 
To be more specific, 
if $f$ is real analytic on $U$
and
the support of $\varphi$ is contained 
in a sufficiently small neighborhood of 
the origin, then
the integral $I(t;\varphi)$ has an
asymptotic expansion of the form
\begin{equation}
I(t;\varphi)\sim
\sum_{\alpha}\sum_{k=1}^{n} 
C_{\alpha k}(\varphi)t^{\alpha} (\log t)^{k-1} \quad 
\mbox{as $t \to +\infty$}, 
\label{eqn:3.1}
\end{equation}
where 
$\alpha$ runs through a finite number of 
arithmetic progressions, not depending on the amplitude 
$\varphi$, which consist of negative rational numbers.    
We are interested in the largest $\alpha$ 
occurring in the asymptotic expansion (\ref{eqn:3.1}). 
Let $S(f)$ be the set 
of pairs $(\alpha,k)$ such that 
for each neighborhood of the origin in $\R^n$, 
there exists a $C^{\infty}$ function $\varphi$ 
with support in this neighborhood 
for which $C_{\alpha k}(\varphi)\neq 0$ in the asymptotic 
expansion (\ref{eqn:3.1}).  
We denote by $(\beta(f), \eta(f))$
the maximum of the set $S(f)$ 
under the lexicographic ordering, 
i.e., $\beta(f)$ is the maximum of
values $\alpha$ for which we can find 
$k$ so that $(\alpha,k)$ belongs to $S(f)$;
$\eta(f)$ is the maximum of integers
$k$ satisfying that $(\beta(f),k)$ belongs to 
$S(f)$. 
We call $\beta(f)$   
the {\it oscillation index}  
of $f$ and 
$\eta(f)$ the {\it multiplicity} of 
its index. 
(This multiplicity, less one, is
equal to the corresponding multiplicity in 
\cite{agv88}, p183.)

The oscillation index and its multiplicity are
precisely estimated or determined by  
Varchenko in \cite{var76} and Arnold, Gusein-Zade and Varchenko 
\cite{agv88}. 
Their investigations need the following condition.
A $C^{\infty}$ function $f$ 
is said to be {\it nondegenerate} 
over $\R$ with respect to the Newton polyhedron 
$\Gamma_+(f)$ if
for every compact face $\gamma$ of $\Gamma_+(f)$, 
the polynomial $f_{\gamma}$ satisfies 
\begin{equation}
\nabla f_{\gamma}=
\left(
\frac{\d f_{\gamma}}{\d x_1},\ldots,
\frac{\d f_{\gamma}}{\d x_n}\right)
\neq (0,\ldots,0) \quad
\mbox{on the set $U\cap(\R\setminus\{0\})^n$}.
\label{eqn:3.2}
\end{equation}

\begin{theorem}[\cite{var76},\cite{agv88}]
Suppose that 
$f$ is real analytic on $U$ and 
is nondegenerate
over $\R$ with respect to its Newton polyhedron. 
Then one has the following:
\begin{enumerate}
\item[(i)]
The progression $\{\a\}$ in $(\ref{eqn:3.1})$
belongs to finitely many 
arithmetic progressions, 
which are obtained by using the theory
of toric varieties based on the geometry of 
the Newton polyhedron $\Gamma_+(f)$. 
$($See Remark~3.4, below.$)$
\item[(ii)] $\beta(f)\leq -1/d(f)$.
\item[(iii)] If at least one of the following three conditions 
is satisfied:
\begin{enumerate}
\item $d(f)>1$;
\item $f$ is nonnegative or nonpositive on $U$;
\item $1/d(f)$ is not an odd integer and $f_{\tau_*}$ 
does not vanish on $U\cap(\R\setminus\{0\})^n$,
\end{enumerate}
then 
$\beta(f)=-1/d(f)$ and $\eta(f)=m(f)$.
\end{enumerate}
\end{theorem}
\begin{remark}
In the assertion~(iii),
more precise situation for amplitudes is seen as follows. 
If 
$\Re(\varphi(0))>0$ 
(resp. $\Re(\varphi(0))<0$) 
and 
$\Re(\varphi)$ is nonnegative 
(resp. nonpositive) on $U$ 
and the support of $\varphi$ is
contained in a sufficiently small neighborhood 
of the origin,  
then we have 
$\lim_{t\to\infty}t^{1/d(f)}(\log t)^{-m(f)+1} 
\cdot I(t;\varphi)\neq 0$. 
Here $\Re(\cdot)$ expresses the real part. 
\end{remark}

\subsection{Our results}

Let us explain our results. 
They need the following condition:
\begin{enumerate}
\item[(C)] 
$f$ belongs to the class 
$\hat{\mathcal E}(U)$ and 
is nondegenerate over $\R$ 
with respect to its Newton polyhedron. 
\end{enumerate}

Since Hironaka's resolution theorem
requires the hypothesis of the
real analyticity,
the existence of the asymptotic expansion
of $I(t;\varphi)$ is not trivial 
in the smooth phase case. 

\begin{theorem}
If $f$ satisfies the condition $(C)$
and the support of $\varphi$ is
contained in a sufficiently small neighborhood 
of the origin, 
then $I(t;\varphi)$ admits an asymptotic expansion 
of the form $(\ref{eqn:3.1})$, 
where $\{\alpha\}$ belongs to the same progressions 
as in the case that the phase is $f_*$, 
which is the principal part of $f$. 
$($Since $f_*$ is a polynomial, the progressions 
can be exactly constructed as in \cite{var76}.$)$
\end{theorem}

\begin{remark}
To be more specific, 
the above set $\{\alpha\}$ belongs to the following
set:
\begin{equation}
\left\{
-\frac{\langle a\rangle +\nu}{l(a)}
;\,\, \nu\in\Z_+,\,\, a\in\tilde{\Sigma}^{(1)}
\right\}
\cup (-\N),
\end{equation}
where $l(a)$ and 
$\tilde{\Sigma}^{(1)}$ are as in Theorem~9.1  
in Section~9. 
We remark that $l(a)$ and 
$\tilde{\Sigma}^{(1)}$ are determined by the 
geometry of 
$\Gamma_+(f)=\Gamma_+(f_*)$ only. 
\end{remark}

Since the existence of the
asymptotic expansion of the form (\ref{eqn:3.1})
has been shown in the above theorem, 
the oscillatory index $\beta(f)$ and 
its multiplicity $\eta(f)$ 
for a given $f$ satisfying the condition (C) 
are defined 
in a similar fashion 
to the real analytic case. 

The following theorem,
which generalizes the assertion (ii) in Theorem~3.1, 
gives an accurate decay estimate for 
$I(t;\varphi)$
by using the Newton distance $d(f)$ and its multiplicity
$\eta(f)$.

\begin{theorem}
If $f$ satisfies the condition $(C)$ 
and the support of $\varphi$ is
contained in a sufficiently small neighborhood 
of the origin, 
then 
there exists a positive constant $C(\varphi)$ 
depending on $\varphi$ but being independent of 
$t$ such that 
\begin{equation}
|I(t;\varphi)|\leq C(\varphi)t^{-1/d(f)}
(\log t)^{m(f)-1}
\quad \mbox{
for $t \geq 1$.}
\end{equation}
This implies 
$\beta(f)\leq -1/d(f)$.
\end{theorem}
\begin{remark}
The above theorem is not only a generalization to 
(ii) in Theorem~3.1  
but also is a slightly stronger result
even if $f$ is real analytic. 
Indeed, from the argument in \cite{var76},\cite{agv88},
the estimate 
$|I(t;\varphi)|\leq \tilde{C}(\varphi)t^{-1/d(f)}
(\log t)^{m(f)}$ for $t\geq 1$ with 
$\tilde{C}(\varphi)>0$
is obtained, but 
more delicate computation of 
coefficients in the asymptotic expansion (\ref{eqn:3.1})
can improve this estimate. 
\end{remark}

Next, let us consider the case that
the equations 
$\beta(f)=-1/d(f)$ and 
$\eta(f)=m(f)$ hold.
The following theorem generalizes the assertion (iii) in Theorem~3.1.

\begin{theorem} 
If $f$ satisfies the condition $(C)$ and 
at least one of the following three conditions is satisfied: 
\begin{enumerate}
\item[(a)]
$d(f)>1$;
\item[(b)]
$f$ is nonnegative or nonpositive on $U$;
\item[(c)]
$1/d(f)$ is not an odd integer and 
$f_{\tau_*}$ does not vanish on $U\cap(\R\setminus\{0\})^n$,
\end{enumerate}
then  
the equations 
$\beta(f)=-1/d(f)$ and 
$\eta(f)= m(f)$ hold. 
\end{theorem}
\begin{remark}
Even if the principal face $\tau_*$ is not compact,
the $\tau_*$-part of $f$ is realized as a smooth function
from the condition (C).
\end{remark}
\begin{remark}
The condition of amplitudes, 
which attain the above equalities, is the same as in Remark~3.2. 
\end{remark}
\begin{remark}
Under the hypotheses in the above theorem, 
we will give explicit formulae 
for the coefficient of the leading term of the asymptotic
expansion (\ref{eqn:3.1}) 
(see Theorem~10.1 in Section~10).
Related results have been obtained 
in \cite{sch91} for convex finite line type phases, 
in \cite{dls97},\cite{dns05} for real analytic phases, 
in \cite{gre09} for phases in two dimensions.
\end{remark}
\section{Lemmas on polyhedra}

Every polyhedron treated in this paper 
satisfies a condition in the following lemma. 
\begin{lemma}
Let $P\subset\R^n$ be a polyhedron. 
Then the following conditions are equivalent.
\begin{enumerate}
\item $P+\R_+^n\subset P\subset\R_+^n;$
\item There exists a finite set of pairs 
$\{(a^j,l_j)\}_{j=1}^N\subset \Z_+^n\times \Z_+$ such that
$P=\bigcap_{j=1}^N H^+(a^j,l_j);$
\item There exists a finite set of pairs 
$\{(b^j,m_j)\}_{j=1}^M\subset \Z_+^n\times \Z_+$ such that 
$P=\bigcap_{j=1}^M H^+(b^j,m_j)$ and 
$P\cap H(b^j,m_j)$ is a facet of $P$ 
for all $j$.
\end{enumerate} 
\end{lemma}
\begin{proof}
(i) $\Longrightarrow$ (ii). \quad
Suppose that (ii) does not hold. 
From the definition of the polyhedron, 
$P$ is expressed as 
$P=\bigcap_{j=1}^N H^+(a^j,l_j)$ with 
$(a^j,l_j)\in\Z^n\times \Z$. 
Here, it may be assumed that
the set $A:=\{(a^j,l_j)\}_{j=1}^N$ satisfies 
that $P\cap H(a^j,l_j)\neq\emptyset$ for all $j$.
If $(a,l)\in A$ belongs to $\Z_+^n\times(-\N)$, 
then $P\cap H(a,l)=\emptyset$. 
On the other hand, if there exists $(a,l)\in A$
with $a\in \Z^n\setminus\Z_+^n$, then
the nonempty face $\gamma:=P\cap H(a,l)$ satisfies
$\gamma+\R_+^n\not\subset H^+(a,l)$, 
which implies $P+\R_+^n\not\subset P$.

(ii) $\Longrightarrow$ (i). \quad 
This implication easily follows from the following:
For any $(a,l)\in \Z_+^n\times\Z_+$, 
if $\alpha\in H^+(a,l)$, then
$\alpha+\R_+^n\subset H^+(a,l)$.

(iii) $\Longrightarrow$ (ii). \quad Obvious.

(ii) $\Longrightarrow$ (iii). \quad
This can be shown by 
using the {\it Representation theorem for polytopes} 
in \cite{zie95}\ (Theorem~2.15, p. 65), which can be 
easily generalized to the case of polyhedra. 
\end{proof}

Hereafter in this section, 
we assume that $P+\R_+^n\subset P$. 
For a face $\gamma$ of $P$, 
we define the subsets in $\{1,\ldots,n\}$ as 
\begin{equation}
V(\gamma)
=\{k;\gamma+\R_+ e_k\subset\gamma\}
\mbox{ 
\,\,and \,\,}
W(\gamma)=\{1,\ldots,n\}\setminus V(\gamma),
\label{eqn:4.1}
\end{equation} 
where $e_k=(0,\ldots,1,\ldots,0)$.
\begin{lemma}
Let $k$ be in $\{1,\ldots,n\}$.
Then the following conditions are equivalent.
\begin{enumerate}
\item $k\in V(\gamma)$;
\item There exists a point $\alpha\in\gamma$
such that $\alpha+\R_+e_k\subset\gamma$; 
\item
For any valid pair 
$(a,l)=(a_1,\ldots,a_n,l)$ defining $\gamma$, 
$a_k=0$.
\end{enumerate}
\end{lemma}
\begin{proof}
(i) $\Longrightarrow$ (ii). \quad
Obvious.

(ii) $\Longrightarrow$ (iii). \quad
Let $(a,l)$ be an arbitrary valid pair defining $\gamma$.
Considering that $\a\in H(a,l)$ means 
that $\alpha$ is a solution of 
the equation $\langle a,\alpha\rangle =l$, 
we easily see that the following conditions are equivalent: 
\begin{enumerate}
\item[(i$'$)] $H(a,l)+\R e_k\subset H(a,l);$
\item[(ii$'$)] There exists a point $\alpha\in H(a,l)$
such that $\alpha+\R_+ e_k\subset H(a,l)$; 
\item[(iii$'$)] $a_k=0.$
\end{enumerate}
Since (ii) implies (ii$'$), the desired implication is shown. 

(iii) $\Longrightarrow$ (i). \quad
By using the above equivalences and the condition 
$P+\R_+^n\subset P$,
this implication is shown as follows.
\begin{equation*}
\begin{split}
&\gamma+\R_+e_k\subset (P\cap H(a,l))+\R_+ e_k \\
&\,\,\,\subset (P+\R_+^n)\cap(H(a,l)+\R e_k)
\subset P\cap H(a,l)=\gamma.
\end{split}
\end{equation*}
\end{proof}
As a corollary of the above lemma, we easily obtain the following. 
The proofs are omitted. 
\begin{lemma}
A face $\gamma$ of $P$ is compact if and only if 
$V(\gamma)=\emptyset$. 
\end{lemma}

\begin{lemma}
Let $\gamma$ be a nonempty face of $P$. 
For any valid pair $(a,l)=(a_1,\ldots,a_n,l)$ defining 
$\gamma$, 
the following equations hold:
\begin{equation}
V(\gamma)=\{k;a_k=0\},\quad\quad 
W(\gamma)=\{k;a_k\neq 0\}. \\
\end{equation}
$($These equations mean that
the set $\{k;a_k=0\}$ is independent of 
the chosen valid pair defining $\gamma$.$)$
\end{lemma}


\section{Remarks on the $\gamma$-part}
Throughout this section, we assume that
$f$ is a $C^{\infty}$ function defined on an open neighborhood $U$
of the origin in $\R^n$, 
whose Taylor series is $\sum_{\alpha}c_{\alpha}x^{\alpha}$, 
and $P\subset\R_+^n$ is a polyhedron
containing the Newton polyhedron $\Gamma_+(f)$.

The following lemma is ``Taylor's formula'', 
which is useful for the analysis below.
\begin{lemma}
Let $V,W$ be subsets in $\{1,\ldots,n\}$ such that
the disjoint union of $V$ and $W$ is
$\{1,\ldots,n\}$. 
Then, $f$ can be expressed as follows:
For any $N\in\N$, 
\begin{equation}
f(x)=
\sum_{\alpha\in A_V(N)}
\frac{1}{\alpha!}
(\partial^{\alpha}f)(T_W(x))
x^{\alpha}
+
\sum_{\alpha\in B_V(N)}
R_{\alpha}(x)x^{\alpha} \quad 
\mbox{for $x\in U$},
\label{eqn:5.1}
\end{equation}
where 
$$
R_{\alpha}(x)=
\frac{N}{\alpha !}
\int_0^1 (1-t)^{N-1}
(\partial^{\alpha}f)
(tT_V(x)+T_W(x))dt
$$ 
and 
\begin{equation}
A_V(N):=\{\alpha\in T_V(\Z_+^n);\langle \alpha\rangle <N\}, \quad
B_V(N):=\{\alpha\in T_V(\Z_+^n);\langle \alpha\rangle =N\}. 
\label{eqn:5.2}
\end{equation}
Here
$T_V(\Z_+^n)=\{\alpha\in\Z_+^n;\alpha_j=0 
\,\,\mbox{for $j\in V$}\}$ as in $(\ref{eqn:1.3})$. 
\end{lemma}
\begin{proof}
For $\varphi\in C^{\infty}((-\delta,1+\delta))$ 
with $\delta>0$, 
the integration by part gives 
$$
\varphi(1)=
\sum_{k=0}^l\frac{\varphi^{(k)}(0)}{k!}+
\frac{1}{l!}\int_0^1(1-t)^l \varphi^{(l+1)}(t)dt.
$$
Applying 
$\varphi(t)=f(tT_V(x)+T_W(x))$ 
to the above formula, 
we can easily obtain the lemma. 
\end{proof}
Hereafter, let $V:=V(\gamma)$ and $W:=W(\gamma)$. 
We use the following symbols:
\begin{equation}
\begin{split}
&H_V(a,l):=T_V(H(a,l))\cap\Z_+^n \,
(=H(a,l)\cap T_{V}(\Z_+^n));\\
&H^+_V(a,l):=T_V(H^+(a,l))\cap\Z_+^n \,
(=H^+(a,l)\cap T_{V}(\Z_+^n));\\
&\quad\quad\gamma_V:=T_V(\gamma)\cap\Z_+^n.
\end{split}
\label{eqn:5.3}
\end{equation}
Note that the sets $H_V(a,l)$ and $\gamma_V$ 
are finite.

Using Lemma~5.1, we easily see 
the following fundamental properties of 
the $\gamma$-part.
\begin{proposition}
\begin{enumerate}
\item
If $\gamma$ is a compact face of $P$,
then
$f$ admits the $\gamma$-part on $U$.
\item
If $f$ is real analytic on $U$, then
$f$ admits the $\gamma$-part on $U$
for any nonempty face $\gamma$ of $P$.
\item
If $f$ admits the $\gamma$-part on $U$,
then the limits $(\ref{eqn:2.2})$ take the same value for 
all valid pairs $(a,l)$ defining $\gamma$. 
The value of these limits can be expressed as 
\begin{equation}
f_{\gamma}(x)=
\sum_{\alpha\in \gamma_V}
\frac{1}{\alpha!}
(\partial^{\alpha}f)(T_{W}(x))
x^{\alpha}\quad \mbox{for $x\in U$},
\label{eqn:5.4}
\end{equation}
where $V=V(\gamma)$ and $W=W(\gamma)$.
In particular, if $\gamma$ is a compact face, then
$f_{\gamma}(x)=\sum_{\alpha\in\gamma}c_{\alpha}x^{\alpha}$
for $x\in U$.
\item 
The above $f_{\gamma}$ is a $C^{\infty}$ function defined
on $U$.
$($If $f$ is real analytic on $U$, then so is $f_{\gamma}$.$)$ 
\end{enumerate}
\end{proposition}
\begin{proof}
By using Lemma~5.1, 
we have (\ref{eqn:5.1}),
where $N\in\N$ satisfies the condition
$H_V(a,l)\subset A_V(N)$.
Noticing that the following equation holds:
\begin{equation}
T_W(t^{a_1}x_1,\ldots,t^{a_n}x_n)=
T_W(x) 
\quad \mbox{for $t\in[0,1]$, $x\in U$},
\label{eqn:5.5}
\end{equation}
we have
\begin{equation}
\begin{split}
&f(t^{a_1}x_1,\ldots,t^{a_n}x_n)= 
\sum_{\alpha\in A_V(N)}
\frac{1}{\alpha!}
(\partial^{\alpha}f)(T_W(x))
x^{\alpha}t^{\langle a,\alpha\rangle }\\
&\quad\quad\quad\quad\quad\quad\quad\quad\quad\quad+
\sum_{\alpha\in B_V(N)}
R_{\alpha}(t^{a_1}x_1,\ldots,t^{a_n}x_n)
x^{\alpha} t^{\langle a,\alpha\rangle },
\end{split}
\label{eqn:5.6}
\end{equation}

When $\gamma$ is compact (equivalently $V(\gamma)=\emptyset$), 
the first summation of the right hand side of 
(\ref{eqn:5.6})
is 
$\sum_{\alpha\in A(N)}
c_{\alpha}x^{\alpha}t^{\langle a,\alpha\rangle }$, 
where $A(N)=\{\alpha\in\Z_+^n;\langle \alpha\rangle <N\}$. 
On the other hand, consider the case that 
$f$ is real analytic on $U$ and $\gamma$ is 
a general nonempty face of $P$.  
Since $\partial^{\alpha}f$ is also real analytic 
on $U$ for any $\alpha\in \Z_+^n$, 
it follows from the shape of the Newton polyhedron
of $f$ and the quasianalytic property 
(see (\ref{eqn:6.6}) in Section~6)
that $(\partial^{\alpha}f)(T_W(x))$ in (\ref{eqn:5.6}) 
vanishes on $U$ for $\alpha\in A_V(N)$ satisfying 
$\langle a,\alpha\rangle <l$.
In these two cases, we see that
the limits (\ref{eqn:2.2})  always exist
by using the condition $H_V(a,l)\subset A_V(N)$, 
which shows (i) and (ii). 

Next, suppose that $f$ admits the $\gamma$-part on $U$. 
Since $\langle a,\alpha\rangle >l$ holds for $\alpha\in B_V(N)$ 
from the condition $H_V(a,l)\subset A_V(N)$, 
the existence of the limit (\ref{eqn:2.2}) implies
that 
$(\partial^{\alpha}f)(T_W(x))$ in (\ref{eqn:5.6}) 
must vanish for
$\alpha\in A_V(N)$ satisfying $\langle a,\alpha\rangle <l$.
Moreover, 
we see that the limit (\ref{eqn:2.2}) is equal to
(\ref{eqn:5.4}), which implies (iii). 
From the form of (\ref{eqn:5.4}), 
(iv) is easily obtained.
\end{proof}
By using the expression of $f_{\gamma}$ in (\ref{eqn:5.4}), 
we see the following two properties of $f_{\gamma}$. 
\begin{lemma}
Let $f$ admit the $\gamma$-part $f_{\gamma}$ on $U$ and
let $\gamma$ be a face of $P$.
Then the Taylor series of $f_{\gamma}$ at the origin is
$\sum_{\alpha\in\gamma\cap\Z_+^n}
c_{\alpha}x^{\alpha}$.
In particular, if
$f$ is real analytic on $U$, then
$f_{\gamma}$ is equal to the convergent 
power series $\sum_{\alpha\in\gamma\cap\Z_+^n}
c_{\alpha}x^{\alpha}$ on $U$.
\end{lemma}
\begin{proof}
To prove the above, it suffices to show the following: 
\begin{equation*}
(\partial^{\beta} f_{\gamma})(0)=
\begin{cases}
(\partial^{\beta} f)(0)& 
\quad \mbox{if $\beta\in\gamma\cap\Z_+^n$}, \\
0&
\quad \mbox{if $\beta\in\Z_+^n\setminus\gamma$}.
\end{cases}
\end{equation*}
If $\beta\in\gamma\cap\Z_+^n$, then
\begin{equation*}
\begin{split}
&
(\partial^{\beta} f_{\gamma})(x)
=\sum_{\alpha\in \gamma_V}
\frac{1}{\alpha!}\cdot
(\partial^{T_{V}(\beta)}
\partial^{\alpha}f)(T_W(x))
\cdot 
(\partial^{T_W(\beta)}x^{\alpha})\\
&\,\,\,
=\frac{1}{\alpha!}\cdot
(\partial^{T_{V}(\beta)}
\partial^{T_{W}(\beta)}f)(T_W(x))
\cdot 
\alpha !=
(\partial^{\beta}f)(T_W(x)).
\end{split}
\end{equation*}
Indeed, 
the value of $\partial^{\beta}x^{\alpha}$ at the origin
is $\alpha!$ if $\alpha=\beta$ or it vanishes otherwise.
If $\beta\in\Z_+^n\setminus\gamma$, then
a similar computation gives 
$(\partial^{\beta} f_{\gamma})(0)=0$. 
\end{proof}

\begin{lemma}
Let $f$ admit the $\gamma$-part $f_{\gamma}$ on 
$U$. 
Let $(a,l)=(a_1,\ldots,a_n,l)\in\Z_+^n\times\Z_+$
be a valid pair defining $\gamma$. 
Then we have the following:
\begin{enumerate}
\item $f_{\gamma}$ has the quasihomogeneous property:
\begin{equation}
f_{\gamma}(t^{a_1} x_1,\ldots, t^{a_n} x_n)=t^l f_{\gamma}(x) 
\mbox{\,\, for \,\, 
$t\in (0,1]$ and $x\in U$},
\label{eqn:5.7}
\end{equation}
\item
$f_{\gamma}$ can be uniquely extended to be a 
$C^{\infty}$ function defined on the set 
$
U\cup \left(\bigcup_{|r|<\delta} T_{V(\gamma)}^r(\R^n)\right)
$ with the property $(\ref{eqn:5.7})$, 
where $\delta$ is a positive number. 
This extended function is also denoted by $f_{\gamma}$. 
\end{enumerate}
\end{lemma}
\begin{proof}
The expression (\ref{eqn:5.4}) in Proposition~5.2
and the equation (\ref{eqn:5.5})
give the above quasihomogeneous property in (i).
The assertion (ii) easily follows from (i). 
\end{proof}

\section{Properties of 
$\hat{{\mathcal E}}[P](U)$ and 
$\hat{{\mathcal E}}(U)$}

Throughout this section, 
every polyhedron $P\subset\R_+^n$ always satisfies that 
$P+\R_+^n\subset P$ if $P\neq\emptyset$.
Let $U$ be an open neighborhood of the origin in $\R^n$.

\subsection{Elementary properties}
From the definitions of $\hat{\mathcal E}[P](U)$ and 
$\hat{\mathcal E}(U)$, 
the following properties can be directly seen, 
so we omit the proofs. 
\begin{proposition}
The classes ${\mathcal E}[P](U)$, 
$\hat{\mathcal E}[P](U)$ and 
$\hat{\mathcal E}(U)$
 have the following properties:
\begin{enumerate}
\item
When $n=1$ and $P=[p,\infty)$ with $p\in\Z_+$, 
\begin{enumerate}
\item
${\mathcal E}[P](U)=\hat{\mathcal E}[P](U)
=\{x^{\alpha}\psi(x);\,\alpha-p\in\Z_+, \,\psi\in C^{\infty}(U)\},$
\item
$\hat{\mathcal E}(U)
=\{x^{\alpha}\psi(x); \,\alpha\in\Z_+, \,
\psi\in C^{\infty}(U) \mbox{ with } \psi(0)\neq 0\}$\\
\quad\quad$(=\{f\in C^{\infty}(U);\, \Gamma_+(f)\neq\emptyset\}.)$
\end{enumerate}
\item $\hat{{\mathcal E}}[\R_+^n](U)=
       {\mathcal E}[\R_+^n](U)=C^{\infty}(U).$
\item If $P_1,P_2\subset\R_+^n$ are polyhedra with $P_1\subset P_2$, 
then 
${\mathcal E}[P_1](U)\subset {\mathcal E}[P_2](U)$ and 
$\hat{{\mathcal E}}[P_1](U)\subset \hat{{\mathcal E}}[P_2](U)$. 
\item $C^{\omega}(U)\cap{\mathcal E}[P](U)\subset
       \hat{{\mathcal E}}[P](U)\subset{\mathcal E}[P](U).$
In particular, $C^{\omega}(U)\subset 
       \hat{{\mathcal E}}(U)\subset
       C^{\infty}(U).$
\item
${\mathcal E}[P](U)$ and  
$\hat{\mathcal E}[P](U)$ are ideals of $C^{\infty}(U)$.
\end{enumerate}
\end{proposition}

\begin{remark}
Unfortunately, the class $\hat{\mathcal E}(U)$ is not closed 
in the following sense.
\begin{enumerate}
\item 
(Summation.) \,\,
Consider the following two-dimensional example: 
$f(x)=x_1+x_1^2+x_1e^{-1/x_2^2}$; 
$g(x)=-x_1$. 
It is easy to see that $f,g\in \hat{\mathcal E}(U)$, but
$f+g\not\in \hat{\mathcal E}(U)$.
\item 
(Change of coordinates.) \,\,
Consider the following two-dimensional example:
$f(x_1,x_2)=(x_1-x_2)^2+e^{-1/x_2^2}$. 
The diffeomorphism $x=\psi(y)$ defined around the origin 
is defined by
$x_1=y_1+y_2$ and $x_2=y_2$. 
It is easy to see that $f\in\hat{\mathcal E}(U)$,
but $f\circ \psi \not\in \hat{\mathcal E}(\psi^{-1}(U))$.
\end{enumerate}
\end{remark}
\subsection{Equivalent conditions}

The following is an important characterization of 
the class $\hat{{\mathcal E}}[P](U)$, 
which is considered as a generalization of 
the property (i-a) in Proposition~6.1.
Denote by ${\mathcal S}[P]$ the set of finite sets
in $P\cap \Z_+^n$. 
\begin{proposition} 
If $P$ is a nonempty polyhedron, then
the following conditions are equivalent.
\begin{enumerate}
\item 
$f$ belongs to the class 
$\hat{{\mathcal E}}[P](U)$;
\item 
There exist
$S \in{\mathcal S}[P]$
and $\psi_{p}\in C^{\infty}(U)$ for $p\in S$ 
such that
\begin{equation}
f(x)=\sum_{p\in S} x^{p}\psi_{p}(x).
\label{eqn:6.1}
\end{equation} 
\end{enumerate}
Note that the above expression is not unique. 
\end{proposition}
\begin{proof}
It suffices to show the assertion in the proposition
in the case of the polyhedron: $P\cap H^{+}(a,l)$
for any $(a,l)\in \Z_+^n\times \Z_+$ instead 
of $P$ under the assumption that the assertion is satisfied 
in the case of $P$.
Indeed, 
since every polyhedron is defined as an intersection of 
finitely many closed half spaces, 
an inductive argument gives the proof of the above proposition. 
Note that the case when $P=\R_+^n$ is obvious.

Since the implication (ii) $\Longrightarrow$ (i) 
is easy, 
we only show the 
implication (i) $\Longrightarrow$ (ii). 

Now, let us assume that $f(x)$ can be expressed 
as in (\ref{eqn:6.1}).
Let a pair 
$(a,l)=(a_1,\ldots,a_n,z)\in \Z_+^n\times \Z_+$ 
be fixed. 

Using Lemma~5.1 with $V=V(\gamma)$ and 
$W=W(\gamma)$, 
we have 
\begin{equation}
\psi_p(x)=
\sum_{\alpha\in A_V(N)}
C_{p\alpha}(T_{W(\gamma)}(x))x^{\alpha}
+
\sum_{\alpha\in B_V(N)}
R_{p\alpha}(x)x^{\alpha},
\label{eqn:6.2}
\end{equation}
where $C_{p\alpha}, R_{p\alpha}\in C^{\infty}(U)$ 
and $N\in\N$ satisfies the condition: 
$
H_V(a,l)\subset A_V(N).
$
Substituting (\ref{eqn:6.2}) into (\ref{eqn:6.1}), we have 
\begin{eqnarray*}
&&f(x)=\sum_{p\in S, \alpha\in A_V(N)}
C_{p\alpha}(T_{W(\gamma)}(x))x^{p+\alpha}
+
\sum_{p\in S,\alpha\in B_V(N)}
R_{p\alpha}(x)x^{p+\alpha} 
\label{eqn:}\\
&&\quad\quad =:f_1(x)+f_2(x).\nonumber
\end{eqnarray*} 
If $\alpha\in B_V(N)$, then the relationship 
$
H_V(a,l)\subset A_V(N) 
$
implies $p+\alpha\in H_V^+(a,l)$. 
Therefore, it suffices to show the following:
Under the assumption that the limit
in (\ref{eqn:2.2}) exists, 
the coefficients $C_{p\alpha}(T_{W(\gamma)}(x))$ 
in $f_1(x)$ 
vanish on $U$, if $p+\alpha\not\in H_V^+(a,l)$.

First, 
let us give an estimate for the function $f_2(x)$. 
A simple computation gives
$$
f_2(t^{a_1}y_1,\ldots,t^{a_n}y_n)
=\sum_{p\in S, \alpha\in B_N} 
C_{p\alpha}(t^{a_1}y_1,\ldots,t^{a_n}y_n)
t^{\langle a,p+\alpha\rangle }y^{p+\alpha}
$$
for $y\in U$ and $t\in (0,\delta)$.
The condition:  
$H_V(a,l)\subset A_V(N)$ implies that 
$\langle a,T_{V(\gamma)}(\alpha)\rangle \geq l+\epsilon$
with some positive number $\epsilon$ 
for $\alpha\in B_V(N)$. 
Thus 
$\langle a,p+\alpha\rangle \geq \langle a,p\rangle +l+\epsilon 
\geq l+\epsilon$ hold for $p\in S$ and $\alpha\in B_V(N)$.
From the above equation, 
there exist positive numbers $C$ and $\delta$ 
such that 
\begin{equation}
f_2(t^{a_1}y_1,\ldots,t^{a_n}y_n)
\leq C t^{l+\epsilon}
\label{eqn:6.3}
\end{equation}
for $y\in U$ and $t\in (0,\delta)$.

Next, let us consider the function $f_1(x)$.
Noticing (\ref{eqn:5.5}), 
we have
$$
f_1(t^{a_1}y_1,\ldots,t^{a_n}y_n)
=\sum_{p\in S, \alpha\in A_N} 
C_{p\alpha}(T_{W(\gamma)}(y))
t^{\langle a,p+\alpha\rangle }y^{p+\alpha}.
$$
By using the estimate (\ref{eqn:6.3}), 
the condition: $H_V(a,l)\subset A_V(N)$ implies that
$C_{p\alpha}(T_{W(\gamma)}(y))$ must vanish on the set $U$ 
if $p+\alpha\not\in H^+(a,l)$. 
\end{proof}

Proposition~6.3 implies that 
the class $\hat{\mathcal E}[P](U)$ can be written 
in the form 
$$
\hat{\mathcal E}[P](U)=
\left\{\sum_{p\in S} x^p\psi_p(x);
S\in{\mathcal S}[P],\,\,
\psi_p\in C^{\infty}(U) \mbox{ for $p\in S$}
\right\}.
$$
Next, 
let us consider an analogous problem
in the case of $\hat{\mathcal E}(U)$. 
It seems difficult to express this class in such a simple form. 
For a polyhedron $P\subset\R_+^n$,
denote by ${\mathcal V}(P)$ the
set of vertices of $P$.

\begin{lemma}
If $f$ belongs to the class $\hat{\mathcal E}(U)$,
then $f$ is expressed as   
$f(x)=\sum_{p\in S} x^p\psi_p(x)$, where 
$S\in {\mathcal S}[\Gamma_+(f)]$ and 
$\psi_p\in C^{\infty}(U)$. 
Moreover,
$S$ contains 
${\mathcal V}(\Gamma_+(f))$
and 
$\psi_p(0)\neq 0$ if $p\in{\mathcal V}(\Gamma_+(f))$. 
\end{lemma}
\begin{proof}
The expression is directly obtained from the Proposition~6.3
with the definition of $\hat{\mathcal E}(U)$. 
If $S$ does not contain some vertices
of $\Gamma_+(f)$ or $\psi_p(0)=0$ 
for some vertex $p$, then 
$\Gamma_+(\sum_{p\in S} x^p\psi_p(x))
\subsetneq\Gamma_+(f)$, which is a contradiction. 
\end{proof}
The following lemma is a converse of the above lemma.
\begin{lemma}
Let $P$ be a nonempty polyhedron.
If $f$ belongs to the class $\hat{\mathcal E}[P](U)$, 
which is expressed as in $(\ref{eqn:6.1})$, 
where $S$ contains ${\mathcal V}(P)$ and 
$\psi_p(0)\neq 0$ for $p\in {\mathcal V}(P)$,  
then $f$ belongs to the class $\hat{\mathcal E}(U)$.
\end{lemma}
\begin{proof}
The assumption implies $P=\Gamma_+(f)$, which 
means $f\in\hat{\mathcal E}(U)$. 
\end{proof}
For a polyhedron $P\subset\R_+^n$,
denote by $\check{\mathcal E}[P](U)$
the set of $f\in\hat{\mathcal E}[P](U)$ 
which is expressed as 
$f(x)=\sum_{p\in S}x^p\psi_p(x)$,
where $S\in{\mathcal S}[P]$
satisfies ${\mathcal V}(P)\subset S$ and
$\psi_p\in C^{\infty}(U)$ satisfies
that $\psi_p(0)\neq 0$ if $p\in {\mathcal V}(P)$.
Let $\tilde{\mathcal E}(U)$ be the subset in $C^{\infty}(U)$ 
defined by 
\begin{eqnarray*}
&&\tilde{\mathcal E}(U):=
\left\{
\sum_{p\in S}x^p\psi_{p}(x);
S\in{\mathcal S}[\R_+^n],\, 
\psi_{p}\in C^{\infty}(U)
\mbox{ with $\psi_{p}(0)\neq 0$ for $p\in S$}
\right\}.
\end{eqnarray*}
In order to understand the structure of
the class $\hat{\mathcal E}(U)$, 
we express or compare this class
by using the relatively simple classes
$\check{\mathcal E}[P](U)$ and 
$\tilde{\mathcal E}(U)$.

\begin{proposition}
\begin{enumerate}
\item $\hat{\mathcal E}(U)=
\bigcup_P\check{\mathcal E}[P](U),$
where the union is with respect to
all nonempty polyhedra $P$ in $\R_+^n$.
\item
$\tilde{\mathcal E}(U)
\subset \hat{\mathcal E}(U)$. 
More precisely, 
\begin{enumerate}
\item When $n=1$ or $2$,  
$\tilde{\mathcal E}(U)=\hat{\mathcal E}(U);$ 
\item When $n\geq 3$,
$\tilde{\mathcal E}(U) 
\subsetneq\hat{\mathcal E}(U)$. 
\end{enumerate}
\end{enumerate}
\end{proposition}
\begin{proof}
The equation in (i) and the inclusion
$\tilde{\mathcal E}(U) \subset \hat{\mathcal E}(U)$ in (ii)
easily follow from Lemmas~6.4 and 6.5. 
Let us show the properties (a), (b) in (ii). 

(a) \quad
The case when $n=1$ easily follows from the Proposition~6.1 (i-b). 

Consider the case when $n=2$.
Let $f$ belong to the class $\hat{\mathcal E}(U)$, 
which is expressed as in Lemma~6.4.
When $p\in S\setminus {\mathcal V}(\Gamma_+(f))$, 
Taylor's formula
implies that the term $x^p\psi_p(x)$ 
can be written in the form: 
$x^p\psi_p(x)=$ a polynomial   
$
+\sum_{\alpha\in{\mathcal V}(\Gamma_+(f))} 
x^{\alpha}\psi_{p\alpha}(x),
$
where $\psi_{p\alpha}\in C^{\infty}(U)$ with 
$\psi_{p\alpha}(0)=0$. 
Notice that 
$
(\Gamma_+(f)\cap\Z_+^n)\setminus 
\bigcup_{p\in{\mathcal V}(\Gamma_+(f))}(p+\R_+^n)
$
is a finite set in the two-dimensional case. 
By substituting the above into the expression in Lemma~6.4, 
$f$ can be written in the form:
$f(x)=$ a polynomial 
$+\sum_{p\in{\mathcal V}(\Gamma_+(f))} x^p\tilde{\psi}_p(x)$,
where $\tilde{\psi}_p\in C^{\infty}(U)$ with
$\tilde{\psi}_p(0)\neq 0$.
This means that $f$ belongs to the class $\tilde{\mathcal E}(U)$.

(b) \quad 
When $n\geq 3$, consider the example:
$g(x)=x_1^2+\cdots+x_{n-1}^2+x_1x_2 e^{-1/x_n^2}.$
It is easy to see that $g\in \hat{\mathcal E}(U)$, but
$g\not\in\tilde{\mathcal E}(U)$.
\end{proof}

Using Proposition~6.3, we give another expression 
of $f_{\gamma}$. 
Compare to (\ref{eqn:5.4}) in Proposition~5.2.

\begin{lemma} 
Let $f$ belong to $\hat{\mathcal E}(U)$,
which is expressed as $(\ref{eqn:6.1})$ in 
Proposition~6.3, and 
$\gamma$ be a nonempty face of $P$. 
Then $f_{\gamma}$ can be expressed as 
\begin{equation}
f_{\gamma}(x)
=\sum_{p\in\gamma\cap S} x^p\psi_p(T_{W(\gamma)}(x))
\quad\mbox{ for $x\in U$.} 
\label{eqn:6.4}
\end{equation}
\end{lemma}
\begin{proof}
Let $(a,l)=(a_1,\ldots,a_n,l)\in\Z_+^n\times\Z_+$ be 
a valid pair defining 
$\gamma$. 
Noticing that $a_k>0$ if and only if $k\in W(\gamma)$, 
Proposition~6.3 gives 
\begin{equation*}
\begin{split}
&f_{\gamma}(x)=\lim_{t\to 0} 
\frac{f(t^{a_1} x_1,\ldots,t^{a_n} x_n)}{t^l}\\
&\quad=
\lim_{t\to 0} \sum_{p\in S} t^{\langle a,p\rangle -l} 
x^p\psi_p(t^{a_1} x_1,\ldots,t^{a_n} x_n)
=\sum_{p\in\gamma\cap S} x^p\psi_p(T_{W(\gamma)}(x)).
\end{split}
\end{equation*}
\end{proof}
When $f\in\hat{\mathcal E}[P](U)$, 
a slightly stronger result for 
the quasihomogeneous property of $f_{\gamma}$ is 
obtained. Compare to Lemma~5.4 (i).
\begin{lemma}
If $f\in\hat{\mathcal E}[P](U)$, 
then the identity $(\ref{eqn:5.7})$ holds 
for all valid pairs $(a,l)$ satisfying
$\gamma\subset H(a,l)$.
\end{lemma}



\subsection{Denjoy-Carleman classes}

Let us discuss the relationship between
the classes $\hat{{\mathcal E}}[P](U),\hat{{\mathcal E}}(U)$ 
and ${\mathcal E}_M(U)$. 
Here ${\mathcal E}_M(U)$ are
the Denjoy-Carleman quasianalytic classes, 
which are interesting classes in $C^{\infty}(U)$
and have been studied from various viewpoints.  
These classes contain all real analytic functions 
but are strictly larger, 
so they also contain functions with non-convergent 
Taylor expansions. 
We briefly explain the Denjoy-Carleman quasianalytic classes 
and their properties.  
Refer to the paper \cite{bm04} by Bierstone and Milman and
the expositive article \cite{thi08} by Thilliez 
for more detailed properties and recent studies about
these classes. 

Let $U$ be an open neighborhood of the origin in $\R^n$ and 
$M=\{M_0,M_1,M_2,\ldots\}$ an increasing sequence of positive 
real numbers, where $M_0=1$. 
Denote by ${\mathcal E}_M(U)$ 
the set consisting of all real-valued
$C^{\infty}$ functions satisfying that
for every compact set $K\subset U$, 
there exist positive constants $A,B$ such that 
\begin{equation}
\left|
\partial^{\alpha}f(x)
\right|\leq 
AB^{\langle \alpha\rangle } \alpha ! M_{\langle \alpha\rangle } \quad
\mbox{for any $x\in K$ and $\alpha\in\Z_+^n$.}
\label{eqn:6.5}
\end{equation}
The class ${\mathcal E}_M(U)$ is said to be {\it quasianalytic},
if all its elements satisfy the following:
\begin{equation}
\mbox{
If $\partial^{\alpha}f(0)=0$ for any $\alpha\in \Z_+^n$,
then $f\equiv 0$ on $U$.
}
\label{eqn:6.6}
\end{equation}
Of course, the set of real analytic functions is quasianalytic. 
We assume that $M=\{M_k\}_{k\in\Z_+}$ satisfies the condition:  
$M$ is logarithmically convex, i.e., 
\begin{equation}
\frac{M_{j+1}}{M_j}\leq\frac{M_{j+2}}{M_{j+1}} 
\quad \mbox{for all $j\in\Z_+$.}
\label{eqn:6.7}
\end{equation}
This condition implies that 
${\mathcal E}_M(U)$ is a ring and 
${\mathcal E}_M(U)$ contains the ring $C^{\omega}(U)$ of
real analytic functions on $U$.  
The Denjoy-Carleman theorem asserts that 
under the hypothesis (\ref{eqn:6.7}),
${\mathcal E}_M(U)$ is quasianalytic 
if and only if 
\begin{equation}
\sum_{j=0}^{\infty}\frac{M_j}{(j+1)M_{j+1}}=\infty.
\label{eqn:6.8}
\end{equation}
If $M$ satisfies the conditions (\ref{eqn:6.7}) and (\ref{eqn:6.8}),
then ${\mathcal E}_M(U)$ is called a 
{\it Denjoy-Carleman class}.

Now, let us show that our classes 
$\hat{{\mathcal E}}[P](U),\hat{{\mathcal E}}(U)$
contain Denjoy-Carleman classes.
For $f\in C^{\infty}(U)$ and 
$I\subset\{1,\ldots,n\}$, 
$f\circ T_I$ can 
be regarded as the $C^{\infty}$ function 
of $(n-\card I)$-variables defined on 
$U_{I}:=U\cap T_I(\R^n)$, 
which is denoted by $f_I$.
For a sequence of positive numbers $M=\{M_j\}_{j\in\Z_+}$
and a nonnegative integer $k$, 
$M^{+k}$ denotes the shifted sequence
$M^{+k}=\{M_{j+k}\}_{j\in\Z_+}$.

\begin{lemma}
If $f$ belongs to a Denjoy-Carleman class
${\mathcal E}_M(U)$, 
then 
\begin{enumerate}
\item 
$\partial^{\alpha}f$ belongs to 
${\mathcal E}_{M^{+\langle \alpha\rangle }}(U)$
for any $\alpha\in \Z_+^n$;
\item  
$f_I$ belongs to ${\mathcal E}_M(U_I)$
for any subset $I$ in $\{1,\ldots,n\}$. 
\end{enumerate}
\end{lemma}
\begin{proof}
Easy.
\end{proof}
\begin{proposition}
If ${\mathcal E}_M(U)$ is a Denjoy-Carleman class,
then 
${\mathcal E}_M(U)\cap{\mathcal E}[P](U)$ is contained in 
$\hat{{\mathcal E}}[P](U)$ and, 
in particular, 
${\mathcal E}_M(U)$ is contained in 
$\hat{{\mathcal E}}(U)$.
\end{proposition}
\begin{proof}
Let $f$ belong to 
${\mathcal E}_M(U)\cap{\mathcal E}[P](U)$, 
let $\gamma$ be an arbitrary proper face of $P$
defined by a valid pair $(a,l)$ and 
let $V=V(\gamma)$ and $W=W(\gamma)$.
Then, from Lemma~5.1, 
$f$ can be expressed as (\ref{eqn:5.1}),
where $N\in\N$ satisfies the condition
$H_V(a,l)\subset A_V(N)$.
It follows from the shape of 
the Newton polyhedron of $f$ that 
if $\langle a,\alpha\rangle <l$, then
$\partial^{\beta}(\partial^{\alpha}f)(0)=0$
for any $\beta\in T_W(\Z_+^n)$. 
Since $(\partial^{\alpha}f)_W \in
{\mathcal E}_{M^{+|\alpha|}}(U_W)$ from Lemma~6.9,
the quasianalytic property (\ref{eqn:6.6})
implies $(\partial^{\alpha}f)_W\equiv 0$
on $U_W$ if $\langle a,\alpha\rangle <l$. 
In the same fashion as in the proof of Proposition~5.2, 
we can see that $f$ admits the $\gamma$-part. 
\end{proof}

\subsection{Remarks on the definition of the $\gamma$-part}
We discuss delicate issues on 
the definitions of the $\gamma$-part and 
the class $\hat{{\mathcal E}}[P](U)$. 
Symbols are the same as in Sections~2 and 5.
Let us consider the difference of the following two conditions:
\begin{enumerate}
\item[(a)] For any $x\in U$, the limit (\ref{eqn:2.2}) exists 
for {\it all} valid pairs defining $\gamma$.
\item[(b)] For any $x\in U$, the limit (\ref{eqn:2.2}) exists 
for {\it some} valid pair defining $\gamma$.
\end{enumerate}
Recall that $f$ is said to 
{\it admit the $\gamma$-part}
on $U$ if (a) holds. 
Here, (a) obviously implies (b), but 
(b) may not imply (a). 
Indeed, the following three-dimensional example shows this: 
$f(x)=x_1+x_2^2\exp(-1/x^2_3)$. 
In this case, $\Gamma_+(f)=\{(1,0,0)\}+\R_+^3$.
In the case of the face 
$\gamma=\{(1,0,\alpha_3);\alpha_3\in\R_+\}$, 
the limit (\ref{eqn:2.2}) exists for $a=(1,1,0),l=1$, while
it does not exist for $a=(3,1,0),l=3$. 

When $\gamma$ is compact, both (a) and (b) always hold
from the proof of Proposition~5.2. 
Moreover, if $\gamma$ is a facet of $P$, 
then the above (a) and (b) are equivalent.  
Indeed, 
if $(a,l)$ is some valid pair defining $\gamma$, 
then
every valid pair defining $\gamma$ is expressed as
$(ca,cl)$ with $c>0$.
These facts imply that the equivalence of (a) and (b) 
always holds 
in the two-dimensional case because every noncompact
face is a facet.

Next, let us consider the definition of
$\hat{{\mathcal E}}[P](U)$. 
From the proof of Proposition~6.3 and 
the above argument, the equivalence of (ii) and (iii)
in Lemma~4.1 implies that
``any face'' can be replaced by 
``any noncompact facet'' in the definition 
of $\hat{{\mathcal E}}[P](U)$ in Section~2.4.
Therefore, even if
(a) is replaced by (b), this exchange does not
affect the definition of $\hat{{\mathcal E}}[P](U)$.

\section{Toric varieties constructed from polyhedra}

Let $P\subset\R_+^n$ 
be a nonempty $n$-dimensional polyhedron 
satisfying $P+\R_+^n \subset P$.
In this section, 
we recall the method to construct 
a toric variety from a given polyhedron $P$. 
Refer to \cite{ful93}, etc.  
for general theory of toric varieties.


\subsection{Cones and fans}
A {\it rational polyhedral cone} 
$\sigma\subset \R^n$ is a cone
generated by finitely many elements of $\Z^n$. 
In other words, 
there are $u_1,\ldots,u_k \in \Z^n$ such that  
$$ 
\sigma=\{\lambda_1 u_1+\cdots+\lambda_k u_k \in {\R}^n;
\lambda_1,\ldots,\lambda_k\geq 0\}. 
$$
We say that $\sigma$ is {\it strongly convex} if 
$\sigma\cap(-\sigma)=\{0\}$. 

By regarding a cone as a polyhedron in $\R^n$, 
the definitions of {\it dimension}, {\it face}, 
{\it edge}, {\it facet} for the cone 
are given by the same way as in Section~2.  

The {\it fan} is defined to be a finite collection $\Sigma$ 
of cones in ${\R}^n$ with the following properties:
\begin{itemize}
\item
Each $\sigma\in \Sigma$ 
is a strongly convex rational polyhedral cone;
\item 
If $\sigma\in\Sigma$ and $\tau$ is a face of $\sigma$, then 
$\tau\in \Sigma$;
\item 
If $\sigma,\tau\in\Sigma$, 
then $\sigma\cap\tau$ is a face of each.    
\end{itemize}
For a fan $\Sigma$,
the union $|\Sigma|:=\bigcup_{\sigma\in\Sigma}\sigma$ 
is called the {\it support} of $\Sigma$. 
For $k=0,1,\ldots,n$, we denote by $\Sigma^{(k)}$  
the set of $k$-dimensional cones in $\Sigma$.  
The {\it skeleton} of a cone $\sigma\in\Sigma$ is 
the set of all of its primitive 
integer vectors 
(i.e., with components relatively prime in $\Z_+$)
in the edges of $\sigma$. 
It is clear that the skeleton of 
$\sigma\in\Sigma^{(k)}$ 
generates $\sigma$ itself and 
that the number of the elements of skeleton 
is not less than $k$. 
Thus, the set of skeletons of the cones 
belonging to $\Sigma^{(k)}$ is also expressed 
by the same symbol $\Sigma^{(k)}$.


\subsection{The fan associated with $P$ 
and its simplicial subdivision}

%

We denote by $(\R^n)^*$ the dual space of $\R^n$ 
with respect to the standard inner product. 
For $a\in(\R^n)^*$, 
define
\begin{equation}
l(a)=\min\left\{
\langle a,\alpha\rangle ; \alpha\in P
\right\}
\label{eqn:7.1}
\end{equation}
and 
$\gamma(a)=
\{\alpha\in P;\langle a,\alpha\rangle =l(a)\}
(=H(a,l(a))\cap P)$.
We introduce an equivalence relation $\sim$ 
in $(\R^n)^*$ by $a\sim a'$ 
if and only if $\gamma(a)=\gamma(a')$. 
For any $k$-dimensional face $\gamma$ of $P$, 
there is an equivalence class $\gamma^*$ which is defined by 
\begin{eqnarray}
&&\gamma^*:=
\{a\in (\R^n)^*;\gamma(a)=\gamma \mbox{ and $a_j\geq 0$ 
for $j=1,\ldots,n$}\}
\label{eqn:7.2}\\
&&\quad(=
\{a\in (\R^n)^*;\gamma=H(a,l(a))\cap P 
\mbox{ and $a_j\geq 0$ 
for $j=1,\ldots,n$}\}.)
\nonumber
\end{eqnarray}
Here, $P^*:=0$. 
The closure of $\gamma^*$, denoted by 
$\overline{\gamma^*}$, is expressed as
\begin{equation}
\overline{\gamma^*}=
\{a\in (\R^n)^*;\gamma\subset H(a,l(a))\cap P 
\mbox{ and $a_j\geq 0$ 
for $j=1,\ldots,n$}\}.
\label{eqn:7.3}
\end{equation}
It is easy to see that  
$\overline{\gamma^*}$ is an $(n-k)$-dimensional strongly convex rational 
polyhedral cone in $(\R^n)^*$ and, moreover,  
the collection of  
$\overline{\gamma^*}$ gives a fan $\Sigma_0$, 
which is called the {\it fan associated with
a polyhedron} $P$.  
Note that 
$|\Sigma_0|=\R_+^n$.


It is known that 
there exists a {\it simplicial subdivision} $\Sigma$
of $\Sigma_0$,  
that is, $\Sigma$ is a fan satisfying the following properties:
\begin{itemize}
\item 
The fans $\Sigma_0$ and $\Sigma$ have the same support; 
\item 
Each cone of $\Sigma$ lies in some cone of $\Sigma_0$; 
\item 
The skeleton of any cone belonging to $\Sigma$ can be completed 
to a base of the lattice dual to $\Z^n$.
\end{itemize}


\subsection{Construction of toric varieties}

Let $\Sigma_0$ be the fan associated with 
$P$ and 
fix a simplicial subdivision $\Sigma$ of $\Sigma_0$. 
For an $n$-dimensional cone $\sigma\in\Sigma$, 
let 
$a^1(\sigma),\ldots,a^n(\sigma)$ be the skeleton of 
$\sigma$, ordered once and for all. 
Here, we set the coordinates of the vector $a^j(\sigma)$ as 
$$
a^j(\sigma)=(a^j_1(\sigma),\ldots,a^j_n(\sigma)).
$$
With every such cone $\sigma$, we
associate a copy of $\R^n$ which is denoted by 
$\R^n(\sigma)$.
We denote by
$ 
\pi(\sigma):\R^n(\sigma) \to \R^n
$
the map defined by  
$(x_1,\ldots,x_n)=\pi(\sigma)(y_1,\ldots,y_n)$ 
with
\begin{equation}
x_k= \prod_{j=1}^n y_j^{a_k^j(\sigma)}=
y_1^{a_k^1(\sigma)}\cdots y_n^{a_k^n(\sigma)}, \quad\quad 
k=1,\ldots,n.
\label{eqn:7.4} 
\end{equation}
Let $Y_{\Sigma}$ be the union of $\R^{n}(\sigma)$ for $\sigma$
which are glued along the images of $\pi(\sigma)$. 
Indeed, for any $n$-dimensional cones $\sigma,\sigma'\in\Sigma$, 
two copies $\R^n(\sigma)$ and $\R^n(\sigma')$ can be
identified with respect to a rational mapping:
$\pi^{-1}(\sigma')\circ \pi(\sigma):
\R^n(\sigma)\to \R^n(\sigma')$ 
(i.e., $x\in\R^n(\sigma)$ and $x'\in\R^n(\sigma')$ will coalesce 
if $\pi^{-1}(\sigma')\circ\pi(\sigma):x\mapsto x'$).
Then it is known that
\begin{itemize}
\item
$Y_{\Sigma}$ is an $n$-dimensional 
real algebraic manifold;
\item
The map 
$\pi:Y_{\Sigma}\to\R^n$
defined on each $\R^n(\sigma)$ as 
$\pi(\sigma):\R^n(\sigma)\to\R^n$ is proper. 
\end{itemize}
The manifold $Y_{\Sigma}$ is called the 
(real) {\it toric variety} associated with $\Sigma$.            

The following properties of $\pi(\sigma)$ are useful for 
the analysis in Section~9. 
They can be easily seen, so 
we omit their proofs.

\begin{lemma}
The set of the points in $\R^n(\sigma)$ in which $\pi(\sigma)$ 
is not an isomorphism is a union of coordinate planes. 
\end{lemma}

\begin{lemma} 
The Jacobian of the mapping $\pi(\sigma)$ 
is equal to 
\begin{equation}
J_{\pi(\sigma)}(y)
=\epsilon \prod_{j=1}^n  
y_j^{\langle a^j(\sigma)\rangle -1},
\label{eqn:7.5}
\end{equation}
where $\epsilon$ is $1$ or $-1$.
\end{lemma}


\section{
Toric resolution of singularities 
in the class $\hat{{\mathcal E}}(U)$
}

\subsection{Preliminaries}

Let us show many lemmas 
which play important roles in the construction of 
toric resolutions of singularities 
in the class $\hat{{\mathcal E}}(U)$.
Some of them will
be useful for the analysis of local zeta functions
in Section~9. 

Let us explain symbols which will be used in this subsection.

\begin{itemize}
\item
$P\subset\R_+^n$ is a polyhedron satisfying 
$P+\R_+^n\subset P$; 
\item
$\Sigma_0$ is the fan associated with the polyhedron
$P$; 
\item
$\Sigma$ is a simplicial subdivision of $\Sigma_0$;
\item
$\Sigma^{(n)}$ consists of $n$-dimensional cones in $\Sigma$;
\item 
$a^1(\sigma),\ldots,a^n(\sigma)$ is the skeleton
of $\sigma\in\Sigma^{(n)}$, ordered once and for all;
\item
${\mathcal P}(\{1,\ldots,n\})$
is the set of all subsets in 
$\{1,\ldots,n\}$;
\item
${\mathcal F}(P)$ is the set of 
nonempty faces of $P$;
\item
When $I\in{\mathcal P}(\{1,\ldots,n\})$, 
we write $J:=\{1,\ldots,n\}\setminus I$;
\item
$H(\cdot,\cdot)$, $l(\cdot)$ 
are as in (\ref{eqn:2.1}), (\ref{eqn:7.1}), 
respectively.
\end{itemize}

Let $\sigma\in\Sigma^{(n)}$, 
$\gamma\in {\mathcal F}(P)$ and
$I\in{\mathcal P}(\{1,\ldots,n\})$.
Define
\begin{eqnarray}
&&\gamma(I,\sigma)
:=
\bigcap_{j\in I}
H(a^j(\sigma),l(a^j(\sigma)))\cap P, 
\label{eqn:8.1}\\
&&
I(\gamma,\sigma)
:=
\{j;\gamma\subset 
H(a^j(\sigma),l(a^j(\sigma)))\}.
\label{eqn:8.2}
\end{eqnarray}
Here set $\gamma(\emptyset,\sigma):=P$. 
It is easy to see that $\gamma(I,\sigma)\in{\mathcal F}(P)$ and 
$I(P,\sigma)=\emptyset$.


\begin{lemma} 
For $\sigma\in\Sigma^{(n)}$, 
$\gamma\in {\mathcal F}(P)$,
$I\in{\mathcal P}(\{1,\ldots,n\})$, 
we have the following. 
\begin{enumerate}
\item
$\gamma\subset\gamma(I(\gamma,\sigma),\sigma)$ and
$\dim(\gamma)\leq n-\card I(\gamma,\sigma)$.
\item 
$\gamma=\gamma(I,\sigma)
\Longrightarrow I\subset I(\gamma,\sigma)
\Longrightarrow \dim(\gamma)\leq n-\card I$.
\end{enumerate}
\end{lemma}
\begin{proof}
(i) is directly seen from the definitions of 
$\gamma(I,\sigma)$ and $I(\gamma,\sigma)$.
The first implication in (ii) 
is shown as follows: 
$\gamma=\gamma(I,\sigma)
\Rightarrow
\gamma
=
\bigcap_{j\in I}
H(a^j(\sigma),l(a^j(\sigma)))\cap P
\Rightarrow
\gamma\subset H(a^j(\sigma),l(a^j(\sigma)))
\mbox{ for $j\in I$}
\Rightarrow
I\subset I(\gamma,\sigma)$.
From the inequality in (i), 
the second implication in (ii) is obvious.
\end{proof}

Next, consider the case 
when $\dim(\gamma)=n-\card I(\gamma,\sigma)$. 
Define
\begin{equation}
\Sigma^{(n)}(\gamma):=\{
\sigma\in\Sigma^{(n)};\dim(\gamma)=n-\card I(\gamma,\sigma)
\}.
\label{eqn:8.3}
\end{equation}
Note that $\Sigma^{(n)}(P)=\Sigma^{(n)}$. 

\begin{lemma}
For $\sigma\in\Sigma^{(n)}$, 
$\gamma\in {\mathcal F}(P)$,
$I\in{\mathcal P}(\{1,\ldots,n\})$, 
we have the following. Here $\gamma^*$ is as in 
$(\ref{eqn:7.2})$.
\begin{enumerate}
\item 
$\card I(\gamma,\sigma)=\dim(\gamma^*\cap\sigma).$
\item
$\Sigma^{(n)}(\gamma)
=\{
\sigma\in\Sigma^{(n)};\dim(\gamma^* \cap\sigma)
=\dim (\gamma^*)
\}\neq\emptyset.
$
\item 
If $\sigma\in\Sigma^{(n)}(\gamma)$, then
$\gamma=\gamma(I(\gamma,\sigma),\sigma)$.
\end{enumerate}
\end{lemma}

\begin{proof}
(i)\quad
This equation follows from the following equivalences.
$$
j\in I(\gamma,\sigma)
\Leftrightarrow
\gamma\subset H(a^j(\sigma),l(a^j(\sigma)))
\Leftrightarrow
a^j(\sigma)\in \overline{\gamma^*}
\Leftrightarrow
a^j(\sigma)\in 
\overline{\gamma^*}\cap\sigma,
$$
where $\overline{\gamma^*}$ denotes
the closure of ${\gamma^*}$. 
Note that the second equivalence follows from 
(\ref{eqn:7.3}).

(ii)\quad
Putting the equation in (i) and $\dim(\gamma^*)=n-\dim(\gamma)$
together, we see the equality of the sets.
Since the support of the fan $\Sigma$ is $\R_+^n$, 
there exists $\sigma$ such that 
$\dim(\gamma^* \cap\sigma)
=\dim (\gamma^*)$, which implies 
$\Sigma^{(n)}(\gamma)\neq\emptyset$.  

(iii)\quad 
Lemma~8.1 (i),(ii) and 
the assumption imply 
$\dim(\gamma)=\dim(\gamma(I(\gamma,\sigma),\sigma))$. 
In fact, 
$
\dim(\gamma)\leq\dim(\gamma(I(\gamma,\sigma),\sigma))
\leq n-\card I(\gamma,\sigma)=\dim(\gamma).
$
Since $\gamma\subset\gamma(I(\gamma,\sigma),\sigma)$
from Lemma~8.1 (i),  
the above dimensional equation yields 
$\gamma=\gamma(I(\gamma,\sigma),\sigma)$. 
\end{proof}

\begin{remark}
It follows from Lemma~8.2 (iii) that 
the map $\gamma:{\mathcal P}(\{1,\ldots,n\})\times\Sigma^{(n)}
\longrightarrow{\mathcal F}(P)$ is surjective. 
\end{remark}

Hereafter in this subsection, 
we always assume that $\sigma\in\Sigma^{(n)}$, 
$\gamma\in {\mathcal F}(P)$,
$I\in{\mathcal P}(\{1,\ldots,n\})$ have the relationship
$\gamma(I,\sigma)=\gamma\in{\mathcal F}(P)$. 

\begin{lemma}
The pair 
$(\sum_{j\in I}a^j(\sigma),\sum_{j\in I}l(a^j(\sigma)))$
is valid for $P$ and defines the face $\gamma$.
\end{lemma}
\begin{proof}
It follows from the following equivalences that
the pair in the lemma is valid for $P$. 
\begin{equation*}
\begin{split}
&\quad\quad\mbox{$P\subset  
H_+(a^j(\sigma),l(a^j(\sigma)))$ 
for any $j \in I$} \\
&\Longleftrightarrow
\mbox{If $\alpha\in  P$,  then 
$\langle a^j(\sigma),\alpha\rangle \geq l(a^j(\sigma))$ for any $j \in I$} \\
&\Longleftrightarrow
\mbox{If $\alpha\in  P$, then  
$\langle \sum_{j\in I}a^j(\sigma),\alpha\rangle \geq\sum_{j\in I}l(a^j(\sigma))$} \\
&\Longleftrightarrow
\mbox{$P\subset 
H(\sum_{j\in I}a^j(\sigma),\sum_{j\in I}l(a^j(\sigma)))$}.
\end{split}
\end{equation*}
Note that the second equivalence follows 
from the definition of $l(\cdot)$.
Moreover, it is similarly shown that the above pair in the lemma
defines the face $\gamma$, so we omit its proof.    
\end{proof}

\begin{lemma}
For any 
$(I,\sigma)\in{\mathcal P}(\{1,\ldots,n\})\times\Sigma^{(n)}$
$($satisfying $\gamma(I,\sigma)=\gamma)$,  
the subset
$\{k; a_k^j(\sigma)=0 
\mbox{\, for any $j\in I$}\}$ 
in 
$\{1,\ldots,n\}$ is equal to $V(\gamma)$ 
defined as in $(\ref{eqn:4.1})$.   
$($This means that the above subset is independent
of the chosen pair $(I,\sigma)$ satisfying  
$\gamma(I,\sigma)=\gamma$.$)$
\end{lemma}
\begin{proof}
This follows from Lemma~4.4 and Lemma~8.4. 
\end{proof}


Let us consider the following two subsets in $\R^n$. 
\begin{eqnarray}
&&
T_I(\R^n)=\{y\in\R^n; y_j=0 \mbox{ if $j\in I$}\} \,\,\,
\mbox{(as in (\ref{eqn:1.3}))},\nonumber\\
&&
T_I^*(\R^n):=\{y\in\R^n; y_j=0 \mbox{ if and only if $j\in I$}\}
\label{eqn:8.4} \\
&&\quad\quad\quad
(=\{y\in T_I(\R^n); y_j\neq 0 \mbox{ if $j\not\in I$}\}) 
\nonumber.
\end{eqnarray}
The following are important equivalent conditions 
of the compactness of a face.
\begin{proposition}
The following conditions are equivalent.
\begin{enumerate}
\item $\gamma$ is compact;
\item $\sum_{j\in I}a_k^j(\sigma)>0$
for $k=1,\ldots,n$;
\item $V(\gamma)=\emptyset$;
\item $\pi(\sigma)(T_{I}(\R^n))=0$;
\item $\pi(\sigma)(T_{I}^*(\R^n))=0$.
\end{enumerate}
\end{proposition}
\begin{proof}
The equivalence of three conditions 
(i),(ii),(iii) follows from Lemmas~4.2, 4.3 and 8.4.
An easy computation implies that 
$(x_1,\ldots,x_n)=
(\pi(\sigma)\circ T_I)(y_1,\ldots,y_n)$, 
where 
\begin{equation}
x_k:=\begin{cases}
\prod_{j=1}^n y_j^{a_k^j(\sigma)}=
\prod_{j\in J} y_j^{a_k^j(\sigma)}&
\quad \mbox{for $k\in V(\gamma)$}, \\
0& 
\quad \mbox{for $k\in W(\gamma)$}.
\end{cases}
\label{eqn:8.5}
\end{equation}
The equivalence of three conditions 
(iii), (iv), (v) follows from the equations 
in (\ref{eqn:8.5}). 
\end{proof}
\begin{lemma}
The following equality as the map from 
$\R^n$ to $\R^n$ holds:
\begin{equation}
\pi(\sigma)\circ T_I=T_{W(\gamma)}\circ \pi(\sigma).
\label{eqn:8.6}
\end{equation}
\end{lemma}
\begin{proof}
This follows from (\ref{eqn:8.5}) and a computation of 
$T_{W(\gamma)}\circ \pi(\sigma)$.
\end{proof}

Hereafter we assume that 
$f$ belongs to the class 
$\hat{{\mathcal E}}(U)$ and set $P=\Gamma_+(f)$. 

\begin{lemma}
For any $\sigma\in\Sigma^{(n)}$, 
there exists a $C^{\infty}$ function $f_{\sigma}$ defined on 
the set $\pi(\sigma)^{-1}(U)$ such that 
$f_{\sigma}(0)\neq 0$ and 
\begin{equation}
f(\pi(\sigma)(y))=
\left(\prod_{j=1}^n y_j^{l(a^j(\sigma))}\right)f_{\sigma}(y)
\quad\quad \mbox{for $y\in\pi(\sigma)^{-1}(U)$.}
\label{eqn:8.7}
\end{equation}
\end{lemma}
\begin{proof}
Let $y$ be in $\pi(\sigma)^{-1}(U)$.
Since $f$ belongs to the class $\hat{{\mathcal E}}(U)$,
$f$ can be expressed as in (\ref{eqn:6.1}) 
in Proposition~6.3.
Substituting $x=\pi(\sigma)(y)$ into (\ref{eqn:6.1}), 
we have
\begin{equation*}
f(\pi(\sigma)(y))=
\sum_{p\in S}
\left(\prod_{j=1}^n y_j^{\langle a^j(\sigma),p\rangle }\right)
\psi_p(\pi(\sigma)(y)).
\end{equation*}
Now, define 
\begin{equation}
f_{\sigma}(y):=
\sum_{p\in S}
\left(\prod_{j=1}^n y_j^{\langle a^j(\sigma),p\rangle -l(a^j(\sigma))}\right)
\psi_p(\pi(\sigma)(y)).
\label{eqn:8.8}
\end{equation}
Then we obtain the equation of the form (\ref{eqn:8.7}). 
Noticing 
$\langle a^j(\sigma),p\rangle -l(a^j(\sigma))\in\Z_+$ for all $j$, 
we see that $f_{\sigma}$ is smooth on $\pi(\sigma)^{-1}(U)$.
On the other hand, 
the face $\gamma(\{1,\ldots,n\},\sigma)$ becomes a vertex of
$\Gamma_+(f)$, 
which is denoted by $p(\sigma)$. 
Lemma~6.4 and (\ref{eqn:8.8}) imply that
$p(\sigma)\in S$ and
$f_{\sigma}(0)=\psi_{p(\sigma)}(0)\neq 0$.
\end{proof}
The following equation plays an important role
in the resolution of singularities and the analysis
in Section~9.
\begin{lemma}
\begin{equation}
f_{\gamma}(\pi(\sigma)(y))=
\left(\prod_{j=1}^n y_j^{l(a^j(\sigma))}\right)
f_{\sigma}(T_I(y))
\quad\quad \mbox{for $y\in\pi(\sigma)^{-1}(U)$.}
\label{eqn:8.9}
\end{equation}
\end{lemma}
\begin{proof}
Let $y$ be in $\pi(\sigma)^{-1}(U)$.
From Lemma~6.7, we have 
\begin{eqnarray}
&&f_{\gamma}(\pi(\sigma)(y))=
\sum_{p\in \gamma\cap S}
\left(\prod_{j=1}^n y_j^{\langle a^j(\sigma),p\rangle }\right)
\psi_p((T_{W(\gamma)}\circ\pi(\sigma))(y))\nonumber\\
&&
\quad\quad
=\left(\prod_{j\in I} y_j^{l(a^j(\sigma))}\right)
\sum_{p\in \gamma\cap S}
\left(\prod_{j\in J} y_j^{\langle a^j(\sigma),p\rangle }\right)
\psi_p((T_{W(\gamma)}\circ\pi(\sigma))(y)).
\label{eqn:8.10}
\end{eqnarray}
On the other hand, the definition of $f_{\sigma}$ 
in (\ref{eqn:8.8}) gives 
\begin{eqnarray}
&&f_{\sigma}(T_I(y))=
\sum_{p\in \gamma\cap S}
\left(\prod_{j\in J} y_j^{\langle a^j(\sigma),p\rangle -l(a^j(\sigma))}\right)
\psi_p((\pi(\sigma)\circ T_I)(y))\nonumber\\
&&\quad\quad
=\left(\prod_{j\in J} y_j^{l(a^j(\sigma))}\right)^{-1}
\sum_{p\in \gamma\cap S}
\left(\prod_{j\in J} y_j^{\langle a^j(\sigma),p\rangle }\right)
\psi_p((\pi(\sigma)\circ T_I)(y)).
\label{eqn:8.11}
\end{eqnarray}
Putting (\ref{eqn:8.6}), (\ref{eqn:8.10}), (\ref{eqn:8.11}) 
together, 
we get the equation in the lemma. 
\end{proof}

\subsection{Resolution of singularities}

The purpose of this subsection is 
to show the following theorem.

\begin{theorem}
Let
$f$ belong to the class $\hat{\mathcal E}(U)$, 
where $U$ is an open neighborhood of the origin 
in $\R^n$,
let $\Sigma$ be a simplicial subdivision of the fan 
$\Sigma_0$ associated with the Newton polyhedron 
$\Gamma_+(f)$ and 
let $\sigma$ be an $n$-dimensional cone in $\Sigma$, 
whose skeleton is $a^1(\sigma),\ldots,a^n(\sigma)\in \Z_+^n$.
Then 
there exists a $C^{\infty}$ function $f_{\sigma}$ 
defined on the set $\pi(\sigma)^{-1}(U)$ such that
$f_{\sigma}(0)\neq 0$ and 
\begin{equation}
 (f\circ \pi(\sigma))(y)
=\left(\prod_{j=1}^n y_j^{l(a^j(\sigma))}\right) 
f_{\sigma}(y)\quad \mbox{ for $y\in \pi(\sigma)^{-1}(U)$.}
\label{eqn:8.12}
\end{equation} 

Furthermore,
if $f$ is nondegenerate over $\R$ 
with respect to $\Gamma_+(f)$ and  
a subset $I\subset \{1,\ldots,n\}$ satisfies  
$\pi(\sigma)(T_I^*(\R^n))=0$,  
then the set $\{y\in T_I^*(\R^n);f_{\sigma}(y)=0\}$
is nonsingular 
$($the definition of $T_I^*(\R^n)$ was given in $(\ref{eqn:8.4}))$, 
i.e., 
the gradient of the restriction of the function 
$f_{\sigma}$ to $T_I^*(\R^n)$ does not vanish at the points of the set  
$\{y\in T_I^*(\R^n);f_{\sigma}(y)=0\}$. 
\end{theorem}
Consider a toric variety $Y_{\Sigma}$ and 
the map $\pi:Y_{\Sigma}\to\R^n$, which 
are constructed as in Section~7 when 
$P=\Gamma_+(f)$. 
The above theorem shows that 
if $f\in\hat{\mathcal E}(U)$ satisfies the nondegeneracy
condition, then
this map $\pi:Y_{\Sigma}\to\R^n$ 
is a real resolution of singularities.
Indeed, the set $\pi(\sigma)^{-1}(0)$ is expressed 
as a disjoint union of $T_I^*(\R^n)$ for some subsets $I$ 
in $\{1,\ldots,n\}$.

\begin{remark}
Let $b$ be a point on $T_I^*(\R^n)$ satisfying 
$f_{\sigma}(b)=0$.   
By the implicit function theorem, 
there exist local coordinates around $b$ in which 
$f\circ \pi(\sigma)$ 
can be expressed in a normal crossing form. 
To be more specific, 
there exists a local diffeomorphism $\Phi$ defined around $b$
such that 
$y=\Phi(u)$ with $b=\Phi(b)$ and 
\begin{equation}
(f\circ\pi(\sigma)\circ \Phi)(u)=
(u_p-b)\prod_{j\in I}u_j^{l(a_j(\sigma))},
\label{eqn:8.13}
\end{equation}
where $y_j=u_j$ for $j\in I$ and 
$p\in \{1,\ldots,n\}\setminus I$. 
\end{remark}

\begin{proof}[Proof of Theorem 8.10]
Lemma~8.8 implies the existence of 
a $C^{\infty}$ function $f_{\sigma}$ 
satisfying (\ref{eqn:8.12})
with $f_{\sigma}(0)\neq 0$. 
Let us show the rest of the theorem.  

Let $\sigma$ be as in the theorem and 
$I$ a subset in $\{1,\ldots,n\}$ satisfying
$\pi(\sigma)(T_I^*(\R^n))=0$.
Note that $\gamma=\gamma(\sigma,I)$ is a compact face 
from Proposition~8.6. 

Since $\gamma=\gamma(\sigma,I)$, we have
$\gamma\subset H(a^j(\sigma),l(a^j(\sigma)))$ 
for $j\in I$ from Lemma~8.1. Thus, 
Lemma~6.8 implies 
\begin{equation*}
f_{\gamma}(t^{a_1^j(\sigma)}x_1,\ldots,t^{a_n^j(\sigma)}x_n)=
t^{l(a^j(\sigma))}f_{\gamma}(x)
\quad \mbox{ for $j\in I$}.
\end{equation*}
Taking the derivative in (\ref{eqn:8.9}) 
with respect to $t$ and putting $t=1$, 
we obtain Euler's identities:
\begin{equation}
\sum_{k=1}^n 
a_k^j(\sigma)x_k\frac{\partial f_{\gamma}}{\partial x_k}(x)
=
l(a^j(\sigma))f_{\gamma}(x)
\quad \mbox{ for $j\in I$}. 
\label{eqn:8.14}
\end{equation}
On the other hand, 
taking the partial derivative with respect to $y_j$
for $j\in J$ and putting $x=\pi(\sigma)(y)$, 
we have 
\begin{equation}
\begin{split}
&\sum_{k=1}^n 
a_k^j(\sigma)x_k\frac{\partial f_{\gamma}}{\partial x_k}(x)\\
&\quad=
\left(\prod_{i=1}^n y_i^{l(a^i(\sigma))}\right)
\left[
l(a^j(\sigma))(f_{\sigma}\circ T_I)(y)+
y_j\frac{\partial}{\partial y_j} (f_{\sigma}\circ T_I)(y)
\right]\quad
\mbox{ for $j\in J$.}
\end{split}
\label{eqn:8.15}
\end{equation}
Now, let us assume that there exists a point 
$b \in T_I^*(\R^n)$ such that 
$$
f_{\sigma}(b)=
\frac{\partial f_{\sigma}}{\partial y_j}(b)=0 
\quad \mbox{for $j\in J$}.
$$
Then the set
$
U_I(b)
=\{x\in U; x=\pi(\sigma)(T_I^{r}(b)) 
\mbox{ for $r\in\R\setminus\{0\}$} 
\}
$
is contained in $(\R\setminus \{0\})^n$. 
Since $f_{\gamma}$ vanishes on the set
$U_I(b)$ from (\ref{eqn:8.9}), 
the equations (\ref{eqn:8.14}),(\ref{eqn:8.15}) give 
\begin{equation}
\sum_{k=1}^n 
a_k^j(\sigma)x_k\frac{\partial f_{\gamma}}{\partial x_k}(x)
=0 \quad \mbox{ for $x\in U_I(b)$, \, $j=1,\ldots,n$.}
\label{eqn:8.16}
\end{equation}
Since the determinant of the $n\times n$ matrix 
$(a_k^j(\sigma))_{1\leq j,k\leq n}$ is equal to 
$1$ or $-1$, 
this matrix is invertible.
Therefore, we have 
$$
\frac{\partial f_{\gamma}}{\partial x_k}(x)=0 
\quad 
\mbox{ for $x\in U_I(b)$, \, $k=1,\ldots,n$},
$$
which is a contradiction to the nondegeneracy condition 
of $f$ in (\ref{eqn:3.2}).  
\end{proof}


\section{Poles of local zeta functions}

Let $U$ be an open neighborhood of the origin. 
Throughout this section, 
the functions $f$, $\varphi$ always satisfy 
the conditions (A), (B) 
in the beginning of Section 3. 

We investigate the properties of 
poles of the functions:
\begin{equation}
Z_{\pm}(s;\varphi):=\int_{\R^n} f(x)_{\pm}^s \varphi(x)dx, 
\label{eqn:9.1}
\end{equation}
where $f(x)_+=\max\{f(x),0\}$,
$f(x)_-=\max\{-f(x),0\}$ and the {\it local zeta function}:
\begin{equation}
Z(s;\varphi)=\int_{\R^n} |f(x)|^s \varphi(x)dx.
\label{eqn:9.2}
\end{equation}
Note that the above functions have a simple relationship:
$Z(s;\varphi)=Z_+(s;\varphi)+Z_-(s;\varphi)$. 
Since 
$Z_{\pm}(s;\varphi)$ can be expressed as 
\begin{equation}
Z_{\pm}(s;\varphi)
=\sum_{\theta\in\{-1,1\}^n}
\int_{\R_+^n}
f(\theta x)_{\pm}^s \varphi(\theta x)
dx,
\label{eqn:9.3}
\end{equation}
where $\theta x=(\theta_1 x_1,\ldots,\theta_n x_n)$,
we substantially investigate the properties
of the functions:
\begin{equation}
\tilde{Z}_{\pm}(s;\varphi):=
\int_{\R_+^n} f(x)_{\pm}^s \varphi(x)dx.
\label{eqn:9.4}
\end{equation}

It is easy to see that the above functions 
are holomorphic functions in the region 
${\rm Re} (s)>0$.  
For the moment, 
suppose that $f$ is real analytic near the origin. 
It is known (c.f.  \cite{jea70},\cite{mal74}) that 
if the support of $\varphi$ is sufficiently small, 
then the functions $Z_{\pm}(s;\varphi)$ and 
$Z(s;\varphi)$ can be analytically 
continued to the complex plane as 
meromorphic functions and their poles belong 
to finitely many arithmetic progressions 
constructed from negative rational numbers.
(In this section, this kind of process on 
analytic extension often appears.  
We denote by the same symbols   
these extended meromorphic functions 
defined on the complex plane.)
More precisely, Varchenko \cite{var76} describes
the positions of candidate poles of 
these functions and their orders 
by using the theory of toric varieties
based on the geometry of Newton polyhedra. 
His works have been deeply developed 
in \cite{ds89},\cite{ds92},\cite{dls97},\cite{dns05}.

The purpose of this section is to generalize 
the above Varchenko's results to the case 
that the function $f$ belongs to the class $\hat{\mathcal E}(U)$.
The results in this section need the following assumption 
stated as in Section~3. 

\begin{enumerate}
\item[(C)] 
$f$ belongs to the class 
$\hat{\mathcal E}(U)$ and 
is nondegenerate over $\R$ 
with respect to its Newton polyhedron. 
\end{enumerate}

In this section, 
we use the following notation.

\begin{itemize}
\item
$\Sigma_0$ is the fan associated with 
$\Gamma_+(f)$; 
\item
$\Sigma$ is a simplicial subdivision of $\Sigma_0$;
\item
$(Y_{\Sigma},\pi)$ is the real resolution associated 
with $\Sigma$; 
\item 
$a^1(\sigma),\ldots,a^n(\sigma)$ is the skeleton
of $\sigma\in\Sigma^{(n)}$, ordered once and for all;
\item
$J_{\pi}(y)$ is the Jacobian of the mapping of $\pi$.
\end{itemize}

\subsection{Candidate poles}
First, let us state our results on the positions 
and the orders of candidate poles of the functions
$Z_{\pm}(s;\varphi)$, $Z(s;\varphi)$. 
\begin{theorem}
Suppose that $f$ satisfies the condition $(C)$.
If the support of $\varphi$ is
contained in a sufficiently small neighborhood 
of the origin, 
then 
the functions $Z_{\pm}(s;\varphi)$ and $Z(s;\varphi)$ 
can be analytically 
continued to the complex plane as 
meromorphic functions,
which are also denoted by the same symbols,
and 
their poles are contained in the set 
\begin{equation}
\left\{
-\frac{\langle a\rangle +\nu}{l(a)}
;\,\, \nu\in\Z_+,\,\, a\in\tilde{\Sigma}^{(1)}
\right\}
\cup (-\N),
\label{eqn:9.5}
\end{equation}
where $l(a)$ is as in $(\ref{eqn:7.1})$ with
$P=\Gamma_+(f)$ and 
$\tilde{\Sigma}^{(1)}=\{a\in\Sigma^{(1)};l(a)>0\}$. 
Moreover, 
the largest element of the first set 
in $(\ref{eqn:9.5})$ is $-1/d(f)$. 
When $Z_{\pm}(s;\varphi)$ and $Z(s;\varphi)$ have poles  
at $s=-1/d(f)$,  
their orders are at most
\begin{equation*}
\begin{cases}
m(f)& \quad \mbox{if $1/d(f,\varphi)$ 
is not an integer}, \\
\min\{m(f)+1, n\}&
\quad \mbox{otherwise}.
\end{cases}
\end{equation*} 
\end{theorem}
\begin{proof}
First, let us show that 
the above assertions also hold in the case of 
the functions $\tilde{Z}_{\pm}(s;\varphi)$
in (\ref{eqn:9.4}).

\underline{\it Step 1.}\,\, 
({\it Decompositions of $\tilde{Z}_{\pm}(s;\varphi)$}.) 
\quad 
For the moment, 
we assume that $s\in\C$ satisfies ${\rm Re}(s)>0$. 
By using the mapping $x=\pi(y)$, 
$\tilde{Z}_{\pm}(s;\varphi)$ are expressed as 
\begin{eqnarray*}
&&
\tilde{Z}_{\pm}(s;\varphi)=\int_{\R_+^n}f(x)_{\pm}^s \varphi(x)dx \\
&&
\quad 
=\int_{\tilde{Y}_{\Sigma}} 
((f\circ\pi)(y))_{\pm}^s (\varphi\circ\pi)(y) 
|J_{\pi}(y)|dy, 
\end{eqnarray*}
where $\tilde{Y}_{\Sigma}:=Y_{\Sigma}\cap \pi^{-1}(\R_+^n)$
and 
$dy$ is a volume element in $Y_{\Sigma}$.    
It is easy to see that 
there exists a set of $C^{\infty}_0$ functions 
$\{\chi_{\sigma}:Y_{\Sigma} \to\R_+; \sigma\in\Sigma^{(n)}\}$ 
satisfying 
the following properties:
\begin{itemize}
\item 
For each $\sigma\in\Sigma^{(n)}$, 
the support of the function $\chi_{\sigma}$ is contained 
in $\R^n(\sigma)$ and 
$\chi_{\sigma}$ identically equals one 
in some neighborhood of the origin. 
\item 
$\sum_{\sigma\in\Sigma^{(n)}}\chi_{\sigma}\equiv 1$ 
on the support of 
$\chi\circ\pi$.  
\end{itemize}
Applying Theorem~8.10 and Lemmas~7.1 and 7.2, we have 
$$
\tilde{Z}_{\pm}(s;\varphi)=\sum_{\sigma\in\Sigma^{(n)}} 
Z_{\pm}^{(\sigma)}(s)$$ 
with  
\begin{equation}
\begin{split}
&Z_{\pm}^{(\sigma)}(s)
=\int_{\R_+^n} ((f\circ\pi(\sigma))(y))_{\pm}^s 
(\varphi\circ\pi(\sigma))(y) 
\chi_{\sigma}(y)|J_{\pi(\sigma)}(y)|dy \\
&
\quad 
=\int_{\R_+^n} 
\left(
\prod_{j=1}^n y_j^{l(a^j(\sigma))}f_{\sigma}(y)
\right)_{\pm}^s 
\left|
\prod_{j=1}^n y_j^{\langle a^j(\sigma)\rangle -1}
\right|\varphi_{\sigma}(y)dy, 
\end{split}
\label{eqn:9.6}
\end{equation}
where 
$\varphi_{\sigma}(y)=
(\varphi\circ\pi(\sigma))(y)\chi_{\sigma}(y)$.

Consider each function $Z^{(\sigma)}_{\pm}(s)$ 
for $\sigma\in\Sigma^{(n)}$. 
We easily see the existence of finite sets of 
$C^{\infty}_0$ functions 
$\{\psi_k:\R^n\to\R_+\}$ and 
$\{\eta_l:\R^n\to\R_+\}$ satisfying the following conditions. 
\begin{itemize}
\item 
The supports of $\psi_k$ and $\eta_l$ are sufficiently small and 
$\sum_k \psi_k + \sum_l \eta_l \equiv 1$ 
on the support of $\varphi_{\sigma}$.
\item 
For each $k$, 
$f_{\sigma}$ is always positive or negative on the support of $\psi_k$. 
\item 
For each $l$, the support of $\eta_l$ intersects the set 
$\{y\in {\rm Supp}(\varphi_{\sigma});f_{\sigma}(y)=0\}$.
\item
The union of the support of $\eta_l$ for all $l$ contains the set 
$\{y\in {\rm Supp}(\varphi_{\sigma});f_{\sigma}(y)=0\}$.
\end{itemize}

By using the functions $\psi_k$ and $\eta_l$, we have
\begin{equation}
Z^{(\sigma)}_{\pm}(s)=
\sum_k I_{\sigma,\pm}^{(k)}(s)+
\sum_l J_{\sigma,\pm}^{(l)}(s),
\label{eqn:9.7}
\end{equation}
with
\begin{equation}
\begin{split}
&
I^{(k)}_{\sigma,\pm}(s)=\int_{\R_+^n} 
\left(
\prod_{j=1}^n 
y_j^{l(a^j(\sigma))}f_{\sigma}(y)
\right)_{\pm}^s 
\left|
\prod_{j=1}^n y_j^{\langle a^j(\sigma)\rangle -1}
\right|\tilde{\psi}_k(y)dy, 
\\
&
J^{(l)}_{\sigma,\pm}(s)
=\int_{\R_+^n} 
\left(
\prod_{j=1}^n y_j^{l(a^j(\sigma))}f_{\sigma}(y)
\right)_{\pm}^s 
\left|
\prod_{j=1}^n y_j^{\langle a^j(\sigma)\rangle -1}
\right|\tilde{\eta}_l(y)dy, 
\end{split}
\label{eqn:9.8}
\end{equation}
where $\tilde{\psi}_k(y)=\varphi_{\sigma}(y)\psi_k(y)$ and 
$\tilde{\eta}_l(y)=\varphi_{\sigma}(y)\eta_l(y)$. 
If the set 
$\{y\in {\rm Supp}(\varphi_{\sigma});f_{\sigma}(y)=0\}$
is empty, then the functions $J^{(l)}_{\sigma,\pm}(s)$ 
do not appear. 

From the viewpoint of the properties of singularities, 
we divide the functions $\tilde{Z}_{\pm}(s;\varphi)$ as 
$\tilde{Z}_{\pm}(s;\varphi)=I_{\pm}(s)+J_{\pm}(s),$ 
with
\begin{equation}
I_{\pm}(s)=\sum_{\sigma\in\Sigma^{(n)}} \sum_k I^{(k)}_{\sigma,\pm}(s),
\quad\,\, 
J_{\pm}(s)=\sum_{\sigma\in\Sigma^{(n)}} \sum_l J^{(l)}_{\sigma,\pm}(s).
\label{eqn:9.9}
\end{equation} 
\underline{\it Step 2.} 
({\it Poles of $I_{\pm}(s)$}.)
\quad
Let us consider the functions $I^{(k)}_{\sigma,\pm}(s)$. 
An easy computation gives 
\begin{equation}
I^{(k)}_{\sigma,\pm}(s)
=
\int_{\R_+^n}
\left(
\prod_{j=1}^n 
y_j^
{l(a^j(\sigma))s+\langle  a^j(\sigma)\rangle -1}
\right)
f_{\sigma}(y)_{\pm}^s\tilde{\psi}_{k}(y)dy.
\label{eqn:9.10}
\end{equation}

The following lemma is useful for analyzing 
the poles of integrals of the above form. 
\begin{lemma}[\cite{gs64},\cite{agv88}]
Let $\psi(y_1,\ldots,y_n;\mu)$ be a $C^{\infty}_0$ function 
of $y$ on $\R^n$ that is an entire function 
of the parameter $\mu\in\C$. 
Then the function 
$$
L(\tau_1,\ldots,\tau_n;\mu)=\int_{\R_+^n}
\left(\prod_{j=1}^n y_j^{\tau_j}\right)
\psi(y_1,\ldots,y_n;\mu) dy_1\cdots dy_n
$$
can be analytically continued at all the complex values of 
$\tau_1,\ldots,\tau_n$ and $\mu$ as a meromorphic function. 
Moreover all its poles are simple and lie on 
$\tau_j=-1,-2,\ldots$ for $j=1,\ldots,n$. 
\end{lemma}
\begin{proof}[Proof of Lemma~9.2]
The lemma is easily obtained by the integration by parts 
(see \cite{gs64},\cite{agv88}).
\end{proof}
By applying Lemma~9.2 to (\ref{eqn:9.10}), 
each $I_{\sigma,\pm}^{(k)}(s)$ can be analytically
continued to the complex plane as a meromorphic function
and their poles are contained in the set 
\begin{equation}
\left\{
-\frac{\langle a^j(\sigma)\rangle +\nu}{l(a^j(\sigma))};
\nu\in\Z_+, j\in B({\sigma})
\right\},
\label{eqn:9.11}
\end{equation}
where 
\begin{equation}
B(\sigma):=\{j;l(a^j(\sigma))\neq 0\}
\subset\{1,\ldots,n\}.
\label{eqn:9.12}
\end{equation}
From (\ref{eqn:9.9}),
$I_{\pm}(s)$ also become meromorphic functions on $\C$ 
and 
their poles are contained in the union of the sets 
(\ref{eqn:9.11})
for all $\sigma\in\Sigma^{(n)}$.

\underline{\it Step 3.} 
({\it Poles of $J_{\pm}(s)$}.) \quad
Let us consider the functions $J^{(l)}_{\sigma,\pm}(s)$.  
By applying Theorem~8.10 
and changing the integral variables as in Remark~8.11, 
$J^{(l)}_{\sigma,\pm}(s)$ can be expressed as follows.
\begin{equation*}
\begin{split}
&J^{(l)}_{\sigma,\pm}(s)=\\
&\int_{\R_+^n} 
\left(
(y_p-b)
\prod_{j \in B_l(\sigma)} y_j^{l(a^j(\sigma))}
\right)^s_{\pm} 
\left|
\prod_{j \in B_l(\sigma)} y_j^{\langle a^j(\sigma)\rangle -1}
\right|\hat{\eta}_l(y_1,\ldots,y_p-b,\ldots,y_n)dy, 
\end{split}
\label{eqn:9.}
\end{equation*}
where $b>0$, 
$B_l(\sigma)\subsetneq\{1,\ldots,n\}$,
$p\in\{1,\ldots,n\}\setminus B_l(\sigma)$ and 
$\hat{\eta}_l\in C^{\infty}_0(\R^n)$ 
with $\hat{\eta}_l(0)\neq 0$.
In a similar fashion to the case of 
$I_{\sigma,\pm}^{(k)}(s)$, we have
\begin{equation}
J^{(l)}_{\sigma,\pm}(s)=
\int_{\R_+^n}
\left(
y_p^s
\prod_{j\in B_l(\sigma)}
y_j^
{l(a^j(\sigma))s+\langle  a^j(\sigma)\rangle -1}
\right)
\hat{\eta}_{l}(y_1,\ldots,\pm y_p,\ldots,y_n)dy.
\label{eqn:9.13}
\end{equation}

By applying Lemma~9.2 to (\ref{eqn:9.13}), 
each $J_{\sigma,\pm}^{(l)}(s)$ can be analytically
continued to the complex plane as a meromorphic function
and their poles are contained in the set 
\begin{equation}
\left\{
-\frac{\langle a^j(\sigma)\rangle +\nu}{l(a^j(\sigma))};
\nu\in\Z_+, j\in \tilde{B}_l(\sigma) 
\right\}\cup(-\N), 
\label{eqn:9.14}
\end{equation}
where $\tilde{B}_l(\sigma)=
\{j\in B_l(\sigma);l(a^j(\sigma))\neq 0\}$.
The necessity of the set $(-\N)$ in (\ref{eqn:9.13}) 
follows from the existence of  
$y_p^s$ in (\ref{eqn:9.13}). 
We remark that  
$y_j^
{l(a^j(\sigma))s+\langle  a^j(\sigma)\rangle -1}$ 
may also induce the poles on $(-\N)$. 
From (\ref{eqn:9.9}), 
$J_{\pm}(s)$ also become meromorphic functions on $\C$
and 
their poles are contained in the union of the sets 
(\ref{eqn:9.14})
for all $\sigma\in\Sigma^{(n)}$.

Now, in order to investigate properties of 
the {\it first} poles of $Z_{\pm}(s)$, we define
\begin{equation}
\tilde{\beta}(f)=\max\left\{
-\dfrac{\langle a\rangle }{l(a)}; 
a\in\tilde{\Sigma}^{(1)} 
\right\}.
\label{eqn:9.15}
\end{equation}

\underline{\it Step 4.}
({\it Geometrical meanings of $\tilde{\beta}(f)$}.)
\quad
Let us 
consider geometrical meanings of 
the quantity $\tilde{\beta}(f)$. 
For $a \in \Sigma^{(1)}$, 
we denote by $q(a)$ 
the point of the intersection of the hyperplane 
$H(a,l(a))$ with the line $\{(t,\ldots,t)\in\R_+^n;t>0\}$, 
where $H(\cdot,\cdot)$ is as in (\ref{eqn:2.1}). 
Then it is easy to see 
$q(a)=(l(a)/\langle a\rangle ,\ldots,l(a)/\langle a\rangle )$. 
Roughly speaking,  
the fact that 
the value of $-\langle a\rangle /l(a)$ is large
means that the point 
$q(a)$ is far from the origin. 
To be more specific, we have the following equivalences: 
For $a\in\tilde{\Sigma}^{(1)}$,  
\begin{equation}
\tilde{\beta}(f)=-\frac{\langle a\rangle }{l(a)} \, \Longleftrightarrow \,
q_*=q(a) \, \Longleftrightarrow \,
q_*\in H(a,l(a)). 
\label{eqn:9.16}
\end{equation}
(The definition of the point $q_*$ was given in Section~2.2.) 
Thus, it easily follows from the definition of $d(\cdot)$ that 
$\tilde{\beta}(f)=-1/d(f)$. 

\underline{\it Step 5.}
({\it Orders of the poles at $\tilde{\beta}(f)$}.)
\quad
For $\sigma\in\Sigma^{(n)}$, let  
$$
A(\sigma)=\left\{
j\in B(\sigma);
\tilde{\beta}(f)=-\frac{\langle a^j(\sigma)\rangle }{l(a^j(\sigma))} 
\right\}\subset\{1,\ldots,n\},
$$
where $B(\sigma)$ is as in (\ref{eqn:9.12}).
From (\ref{eqn:9.9}), 
it suffices to analyze the poles of 
$I_{\sigma,\pm}^{(k)}(s)$ and
$J_{\sigma,\pm}^{(l)}(s)$. 
When these functions have poles at $s=\tilde{\beta}(f)$,
we see the upper bounds of orders of their poles 
at $s=\tilde{\beta}(f)$ 
as follows 
by applying Lemma~9.2 to the integrals 
(\ref{eqn:9.10}),(\ref{eqn:9.13}). 
\begin{center}
\begin{tabular}{l|l} \hline
\textit{$I_{\sigma,\pm}^{(k)}(s)$} & 
       \quad $\card A(\sigma)$ \\ \hline
\textit{$J_{\sigma,\pm}^{(l)}(s)$} & 
       $\min\{\card A(\sigma),n-1\}$ if $\tilde{\beta}(f)\not\in (-\N)$\\
             & $\min\{\card A(\sigma)+1,n\}$ if $\tilde{\beta}(f)\in (-\N)$ 
\\ \hline
\end{tabular}
\end{center}
From the above table, 
in order to obtain the estimates of the orders of
poles in the theorem, 
it suffices to show the following. 
(Here, we need the inequality ``$\leq$'' only
in the lemma below. The equality 
will be needed in Section~9.3.)
\begin{lemma}
$m(f)=
\max\left\{\card A(\sigma)
;\sigma\in\Sigma^{(n)}
\right\}.$
\end{lemma}
\begin{proof}[Proof of Lemma~9.3]
Recall $m(f):=n-\dim(\tau_*)$. 
From the definition of $A(\sigma)$
and (\ref{eqn:9.15}), we have 
\begin{eqnarray*}
&&A(\sigma)
= \{j ; q_*\in H(a^j(\sigma),l(a^j(\sigma)))\} \\
&& \quad=\{j ; 
\tau_*\subset H(a^j(\sigma),l(a^j(\sigma)))\}=I(\tau_*,\sigma). 
\end{eqnarray*}
Here $\tau_*$ is the principal face of $\Gamma_+(f)$, i.e.,
its relative interior contains the point $q_*$, and
$I(\cdot,\cdot)$ is as in (\ref{eqn:8.2}).  
Lemma~8.1 implies that 
$\dim(\tau_*)
\leq n-\card I(\tau_*,\sigma)
= n-\card A(\sigma)$ for any 
$\sigma\in\Sigma^{(n)}$.
On the other hand, Lemma~8.2 implies that 
there exists $\sigma\in\Sigma^{(n)}$ such that 
$\dim (\tau_*)=n-\card A(\sigma)$. 
Since the codimension of $\tau_*$ is $m(f)$, 
we obtain the equality in the lemma.  
\end{proof}

Since $\tilde{Z}_{\pm}(s)=I_{\pm}(s)+J_{\pm}(s)$,
we see that the poles of $\tilde{Z}_{\pm}(s)$ 
have the same properties as in the theorem. 
Finally, 
considering 
the relationships: (\ref{eqn:9.3}) and 
$Z(s)=Z_{+}(s)+Z_{-}(s)$, 
we obtain the theorem. 
\end{proof} 

\subsection{Poles of $J_{\pm}(s)$ on negative integers}

We consider the poles of the functions
$J_{\pm}(s)$ at negative integers in more detail. 

The following lemma is useful 
for computing the coefficients of 
the Laurent expansion explicitly. 
\begin{lemma}
Let $\psi$ be a $C_0^{\infty}$ function on $\R$ and 
$k\in\N$. 
Then 
$$
\lim_{s\to -k} (s+k)
\int_{0}^{\infty}
y^{s} \psi(y)dy 
=
\frac{1}{(k-1)!}\psi^{(k-1)}(0). 
$$
In particular, 
$$
\lim_{s\to -1} (s+1)
\int_{0}^{\infty}
y^s \psi(y)dy 
=\psi(0). 
$$
\end{lemma}
\begin{proof}
The above formula is easily obtained 
by the integration by parts. 
\end{proof}

For $\lambda\in\N$,
define 
\begin{eqnarray*}
&&
A_{\lambda}(\sigma):=\{j\in B(\sigma); 
l(a^j(\sigma))\lambda-\langle a^j(\sigma)\rangle \in\Z_+\},\\
&&
\rho_{\lambda}:=
\min\{\max\{\card A_{\lambda}(\sigma);\sigma\in\Sigma^{(n)}\},n-1\}.
\end{eqnarray*}

The following proposition will be used in 
the computation of the coefficients 
of the asymptotic expansion (\ref{eqn:3.1})
of $I(t;\varphi)$.

\begin{proposition}
Suppose that $f$ satisfies the condition $(C)$.
If the support of $\varphi$ is
contained in a sufficiently small neighborhood 
of the origin, then the orders of poles of 
$J_{\pm}(s)$  
at $s=-\lambda\in(-\N)$ are not 
higher than $\rho_{\lambda}+1$.
In particular, if $\lambda<1/d(f)$, 
then these orders are not higher than $1$.    
Moreover, 
let $a_{\lambda}^{\pm}$ be the coefficients of 
$(s+\lambda)^{-\rho_{\lambda}-1}$ in the Laurent expansions of 
$J_{\pm}(s)$ at $s=-\lambda$, 
respectively,  
then we have $a_{\lambda}^+=(-1)^{\lambda-1}a_{\lambda}^-$ 
for $\lambda\in\N$.
\end{proposition}
\begin{proof}
Let $\lambda\in\N$, 
$l_j,m_j\in\N$ for $j=1,\ldots,n-1$  
and  
$\eta\in C_0^{\infty}(\R^n)$. 
Let $B_{\lambda}$ be the subset in
$\{1,\ldots,n-1\}$ defined by
$
B_{\lambda}
=\{j;
l_j\lambda-m_j+1\in\N\},
$
and 
let $k_j\in\N$ be defined by 
$k_j=l_j\lambda-m_j+1$ for $j\in B_{\lambda}$. 
We define
\begin{equation*}
g_{\lambda}^{\pm}(s)=
\int_{\R_+^n}
\left(
y_n^s
\prod_{j\in B_{\lambda}}
y_j^
{l_j s+m_j-1}
\right)
\eta(y_1,\ldots,y_{n-1},\pm y_n)dy,
\end{equation*}
respectively. 

It easily follows from Lemma~9.2 that
the functions $g_{\lambda}^{\pm}(s)$ can be analytically extended
to $\C$ as meromorphic functions and 
they have at 
$s=-\lambda$ poles of order not higher 
than $\card B_{\lambda}+1$.
Let $b_{\lambda}^{\pm}$ be the coefficients of 
$(s+\lambda)^{-\card B_{\lambda}-1}$ in the Laurent expansions of 
$g_{\lambda}^{\pm}(s)$ at $s=-\lambda$, 
respectively. 

By carefully observing 
the analysis of $J_{\sigma,\pm}^{(l)}(s)$ 
in the proof of Theorem~9.1, 
it suffices to show the equation:
$b_{\lambda}^+=(-1)^{\lambda-1}b_{\lambda}^-$
for $\lambda\in\N$.  
By using Lemma~9.4, a direct computation gives
$
b_{\lambda}^{\pm}=(\pm 1)^{\lambda-1} C_{\lambda}
$, with
\begin{equation*}
\begin{split}
&C_{\lambda}=
\frac{1}{(\lambda-1)!}
\prod_{j\in B_{\lambda}}
\left(
\frac{1}
{l_j\prod_{\nu_j=1}^{k_j-1}
(l_j\lambda-m_j+1-\nu_j)}
\right)\\
&\times
\begin{cases}
(\partial^{\alpha-1}\eta)(0)&
\quad \mbox{if $B_{\lambda}=\{1,\ldots,n-1\}$,}\\
\displaystyle 
\int_{\R_+^{n-\card B_{\lambda}-1}}
\left(
\partial^{\alpha-1}\eta
\right)
\left(
T_{B_{\lambda}\cup\{n\}}(y)
\right)
\prod_{j\not\in B_{\lambda}\cup\{n\}}
dy_j& \quad 
\mbox{otherwise},
\end{cases}
\end{split}
\end{equation*}
where 
$\alpha=(\alpha_1,\ldots,\alpha_n)$ satisfies 
that 
$\alpha_j=k_j$ if $j\in B_{\lambda}$,
$\alpha_n=\lambda$ and 
$\alpha_j=1$ otherwise. 
From the above equation, we see that 
$b_{\lambda}^+=(-1)^{\lambda-1}b_{\lambda}^-$
for $\lambda\in\N$.
\end{proof}

\subsection{The first coefficients} 
We define the subset of important cones in 
$\Sigma^{(n)}$ as follows.
\begin{equation*}
\Sigma_*^{(n)}
:=\{\sigma\in\Sigma^{(n)};
m(f)=\card A(\sigma)\}.
\label{eqn:9.}
\end{equation*}
It follows from Lemma~9.3 that 
$\Sigma_*^{(n)}$ is nonempty. 
From the definition of $m(f)$, 
we can see the following equivalences:
\begin{eqnarray}
&&\sigma\in\Sigma_*^{(n)} \Longleftrightarrow
\dim(\tau_*)=n-\card A(\sigma)\nonumber\\
&&\quad\quad\Longleftrightarrow
\tau_*=\bigcap_{j\in A(\sigma)}
H(a^j(\sigma),l(a^j(\sigma)))\cap\Gamma_+(f).
\label{eqn:9.17}
\end{eqnarray}
Thus, when 
$\sigma\in\Sigma_*^{(n)}$,
$I=A(\sigma)$, 
$\gamma=\tau_*$,
the equation $\gamma(I,\sigma)=\gamma$ holds,  
which is an important condition in Section~8. 

Now, let us compute the coefficients of 
$(s+1/d(f))^{-m(f)}$ in the Laurent expansions of 
$\tilde{Z}_{\pm}(s;\varphi)$. 
Let 
\begin{equation}
\tilde{C}_{\pm}
(=\tilde{C}_{\pm}(f,\varphi))
:=\lim_{s\to-1/d(f)} 
(s+1/d(f))^{m(f)} 
\tilde{Z}_{\pm}(s;\varphi). 
\label{eqn:9.18}
\end{equation}
\begin{proposition}
Suppose that $f$ satisfies the condition $(C)$ and
that at least one of the following
conditions is satisfied. 
\begin{enumerate}
\item $d(f)>1$;
\item There exists a cone 
$\sigma\in\Sigma_*^{(n)}$ such that 
$f_{\sigma}\circ T_{A(\sigma)}$ 
does not vanish on $\R_+^n\cap \pi(\sigma)^{-1}(U)$.
\end{enumerate}
Here the above $f_{\tau_*}$ is considered 
as an extended smooth function defined 
on a wider region as in Lemma~5.4 $(ii)$.
Then we give explicit formulae for coefficients 
$\tilde{C}_{\pm}$: 
$(\ref{eqn:9.22}),(\ref{eqn:9.24}),(\ref{eqn:9.25}),(\ref{eqn:9.26})$
in the proof of this proposition. 
It follows from these formulae that  
if 
$\Re(\varphi(0))>0$ $($resp. $\Re(\varphi(0))<0)$ 
and 
$\Re(\varphi)$ is nonnegative 
$($resp. nonpositive$)$ on $U$,  
then 
$\Re(\tilde{C}_{\pm})$ are nonnegative 
$($resp. nonpositive$)$ and 
$\Re(\tilde{C}_{+}+\tilde{C}_{-})$ is positive
$($resp. negative$)$. 
Here $\Re(\cdot)$ expresses the real part.
\end{proposition}
\begin{proof}
In this proof, 
we use the following notation and symbols 
to decrease the complexity. 
\begin{itemize}
\item $\prod_{j\not\in A(\sigma)}y_j^{a_j}dy_j$ means 
$\prod_{j\not\in A(\sigma)}y_j^{a_j}\cdot
\prod_{j\not\in A(\sigma)}dy_j$ 
with $a_j>0$; 
\item $L_{\sigma}:=\prod_{j\in A(\sigma)}l(a^j(\sigma))^{-1}$;
\item $M_j(\sigma)
:=-l(a^j(\sigma))/d(f)+\langle  a^j(\sigma)\rangle -1$.
\item 
If $a=0$, then the value of $a^{-1/d(f)}$ is defined by $0$.
\end{itemize}
Note that $M_j(\sigma)$ is a nonnegative constant and,   
moreover, $M_j(\sigma)= 0$ if and only if $j\in A(\sigma)$. 

Let us compute the limits $\tilde{C}_{\pm}$ exactly. 
We divide the computation into the following two cases:
$m(f)<n$ and $m(f)=n$. 
After obtaining the formulae 
(\ref{eqn:9.22}),(\ref{eqn:9.24}),
(\ref{eqn:9.25}),(\ref{eqn:9.26}), below,  
we can easily see that 
$\Re(\tilde{C}_{\pm})\geq 0$ and 
$\Re(\tilde{C}_{+}+\tilde{C}_{-})>0$, 
which are as in the theorem.  

\underline{The case: $m(f)<n$.}\quad
First, we consider the case that the hypothesis (i) is satisfied.
Let us explicitly compute the following limits:
\begin{equation*}
\tilde{C}_{\pm}(\sigma):=
\lim_{s\to-1/d(f)} 
(s+1/d(f))^{m(f)} 
Z^{(\sigma)}_{\pm}(s).
\end{equation*} 
Since $\tilde{C}_{\pm}(\sigma)=0$ if $\sigma\not\in\Sigma_*^{(n)}$,
it suffices to consider the case that $\sigma\in\Sigma_*^{(n)}$. 
Considering the equations (\ref{eqn:9.7}) and 
applying Lemma~9.4 to (\ref{eqn:9.10}), (\ref{eqn:9.13}) 
with respect to each $y_j$ for $j\in A(\sigma)$, 
we have
\begin{equation}
\tilde{C}_{\pm}(\sigma)=
\sum_{k} G_{\pm}^{(k)}(\sigma) +
\sum_{l} H_{\pm}^{(l)}(\sigma),
\label{eqn:9.19}
\end{equation}
with 
\begin{equation}
G_{\pm}^{(k)}(\sigma)=L_{\sigma}
\int_{\R_+^{n-m(f)}}
\frac{
\tilde{\psi}_{k}(T_{A(\sigma)}(y))
}{
f_{\sigma}(T_{A(\sigma)}(y))_{\pm}^{1/d(f)}
}
\prod_{j\not\in A(\sigma)}y_j^{M_j(\sigma)}dy_j, 
\label{eqn:9.20}
\end{equation}
\begin{equation}
H_{\pm}^{(l)}(\sigma)=
L_{\sigma}
\int_{\R_+^{n-m(f)}}
\frac{
\hat{\eta}_{l}
(T_{A(\sigma)}(y_1,\ldots,\pm y_p,\ldots,y_n))
}
{y_p^{1/d(f)}} 
\prod_{j\in B_l(\sigma)\setminus A(\sigma)}y_j^{M_j(\sigma)}dy_j,
\label{eqn:9.21}
\end{equation}
where $\tilde{\psi}_{k}$ and $\hat{\eta}_{l}$ are as 
in (\ref{eqn:9.8}), (\ref{eqn:9.13}), 
the summations in (\ref{eqn:9.19}) are taken for all $k$,$l$ 
satisfying 
$T_{A(\sigma)}(\R^n)\cap {\rm Supp}(\psi_k)\neq \emptyset$ 
and $A(\sigma)\subset B_l(\sigma)$. 
We remark that the values of 
$G_{\pm}^{(k)}(\sigma)$ and 
$H_{\pm}^{(l)}(\sigma)$ may depend on 
the cut-off functions $\chi_{\sigma}$, $\psi_k$, $\eta_l$
in Section~9.1. 
We remark that 
if $f_{\sigma}(T_{A(\sigma)}(y))<0$, then
$f_{\sigma}(T_{A(\sigma)}(y))_+^{-1/d(f)}=0$ in (\ref{eqn:9.20}). 
Since $d(f)>1$, 
the integrals in (\ref{eqn:9.21}) are convergent and they are
interpreted as improper integrals. 

In (\ref{eqn:9.20}), (\ref{eqn:9.21}), 
we deform the cut-off functions $\psi_k$ and $\eta_l$ 
as the volume of the support of $\eta_l$ tends to
zero for all $l$.  
Then, it is easy to see that 
the limit of $H_{\pm}^{(l)}(\sigma)$ is zero, while 
that of $\sum_k G_{\pm}^{(k)}(\sigma)$ can be 
computed explicitly. 
Considering the equation (\ref{eqn:9.19}), we have 
\begin{equation*}
\begin{split}
\tilde{C}_{\pm}(\sigma)=
L_{\sigma}
\int_{\R_+^{n-m(f)}}
\frac{
{\varphi}_{\sigma}(T_{A(\sigma)}(y)) 
}{
f_{\sigma}(T_{A(\sigma)}(y))_{\pm}^{1/d(f)}
}
\prod_{j\not\in A(\sigma)}y_j^{M_j(\sigma)} dy_j,
\end{split}
\end{equation*}
where $\varphi_{\sigma}$ is as in (\ref{eqn:9.6}).

Now, 
let us compute the limits $\tilde{C}_{\pm}$ explicitly. 
If the cut-off function $\chi_{\sigma}$ is deformed as 
the volume of the support of $\chi_{\sigma}$ tends to zero, 
then $\tilde{C}_{\pm}(\sigma)$ tends to zero. 
Notice that each $\R^n(\sigma)$ (see Section~7.3)
is densely embedded in $Y_{\Sigma}$ and 
that
$\tilde{C}_{\pm}=
\sum_{\sigma\in\Sigma_*^{(n)}}\tilde{C}_{\pm}(\sigma)$.
Thus,  
for an arbitrary fixed cone $\sigma\in\Sigma_*^{(n)}$, 
we have 
\begin{equation}
\begin{split}
\tilde{C}_{\pm}=
L_{\sigma}
\int_{\R_+^{n-m(f)}}
\frac{
({\varphi}\circ\pi(\sigma))(T_{A(\sigma)}(y)) 
}{
f_{\sigma}(T_{A(\sigma)}(y))_{\pm}^{1/d(f)}
}
\prod_{j\not\in A(\sigma)} 
y_j^{M_j(\sigma)}dy_j.
\end{split}
\label{eqn:9.22}
\end{equation}

Let us give the other formulae of $\tilde{C}_{\pm}$. 
From the condition (\ref{eqn:9.17}), Lemma~8.9 implies 
\begin{equation}
(f_{\tau_*}\circ\pi(\sigma))(T_{A(\sigma)}^1(y))
=
\left(
\prod_{j\not\in A(\sigma)}
y_j^{l(a^j(\sigma))}
\right)
f_{\sigma}(T_{A(\sigma)}(y)).
\label{eqn:9.23}
\end{equation}
By using the above equation, (\ref{eqn:9.22}) 
can be rewritten as 
\begin{equation}
\begin{split}
\tilde{C}_{\pm}=L_{\sigma}
\int_{\R_+^{n-m(f)}}
\frac{
(\varphi\circ\pi(\sigma))(T_{A(\sigma)}(y))
}{
(f_{\tau_*}\circ\pi(\sigma))(T_{A(\sigma)}^1(y))_{\pm}^{1/d(f)}
}
\prod_{j\not\in A(\sigma)} y_j^{\langle  a^j(\sigma) \rangle -1}
dy_j.
\end{split}
\label{eqn:9.24}
\end{equation}

Secondly, we consider the case that the hypothesis (ii) is satisfied. 
In this case, it suffices to deal with the case of  
$G_{\pm}^{(k)}(\sigma)$ only.
Therefore, the limits $\tilde{C}_{\pm}$
can be samely computed as in (\ref{eqn:9.22}) and (\ref{eqn:9.24}),
where $\sigma$ is as in the hypothesis (ii). 
We remark that $C_+$ or $C_-$ is equal to zero in this case. 

\underline{The case: $m(f)=n$.}\quad
In this case, we see that 
$A(\sigma)=\{1,\ldots,n\}$, 
$m(f)=n$ and 
the principal face $\tau_*$ is the point 
$q_*=(d(f),\ldots,d(f))$.
Similar computations give 
the following. 
The first expression, corresponding to (\ref{eqn:9.22}),
is 
\begin{equation}
\tilde{C}_{\pm}=
L_{\sigma}\frac{\varphi(0)}{f_{\sigma}(0)^{1/d(f)}_{\pm}}.
\label{eqn:9.25}
\end{equation}
The second expression, corresponding to (\ref{eqn:9.24}), 
is 
\begin{equation}
\begin{split}
\tilde{C}_{\pm}=
L_{\sigma}
\frac{
\varphi(0)}{
f_{\tau_*}(1,\ldots,1)^{1/d(f)}_{\pm}
}
=L_{\sigma}(d(f)!)^{n/d(f)}\frac{
\varphi(0)}{
(\partial^{q_*}f)(0)_{\pm}^{1/d(f)}
}.
\end{split}
\label{eqn:9.26}
\end{equation}
\end{proof}
\begin{remark}
From the proof of Proposition~9.6, 
we see the following. 
\begin{enumerate}
\item 
The values of 
(\ref{eqn:9.22}),(\ref{eqn:9.24}),(\ref{eqn:9.25}),(\ref{eqn:9.26})
are independent of the chosen cone $\sigma\in\Sigma_*^{(n)}$.
\item 
The integrals in (\ref{eqn:9.22}),(\ref{eqn:9.24})
are convergent. 
\end{enumerate}
\end{remark}
\begin{remark}
Let us consider the case that $\tau_*$ is compact. 
Then 
$\pi(\sigma)\circ T_{A(\sigma)}(\R^n)=0$ from Lemma~8.6. 
More simple formulae 
are obtained as follows.  
\begin{equation}
\begin{split}
\tilde{C}_{\pm}=L_{\sigma}\varphi(0)
\int_{\R_+^{n-m(f)}}
\frac{1}
{(f_{\tau_*}\circ\pi(\sigma))(T_{A(\sigma)}^1(y))_{\pm}^{1/d(f)}}
\prod_{j\not\in A(\sigma)} y_j^{\langle  a^j(\sigma) \rangle -1}
dy_j.
\end{split}
\label{eqn:9.27}
\end{equation}
\end{remark}

\begin{remark}
In \cite{dls97},\cite{dns05}, 
similar formulae of 
$\tilde{C}_{\pm}$  
are obtained in the real analytic phase case. 
Their results do not require the assumptions (i), (ii) 
in Proposition~9.6. 
\end{remark}
Next, let us compute the coefficients of 
$(s+1/d(f))^{-m(f)}$ in the Laurent expansions of 
$Z_{\pm}(s;\varphi)$, $Z(s;\varphi)$. 
Let 
\begin{equation*}
\begin{split}
&C_{\pm}=\lim_{s\to-1/d(f)} 
(s+1/d(f))^{m(f)} 
Z_{\pm}(s;\varphi), \\
&C=\lim_{s\to-1/d(f)} 
(s+1/d(f))^{m(f)} 
Z(s;\varphi),
\end{split}
\end{equation*}
respectively.
\begin{theorem}
Suppose that $f$ satisfies the condition $(C)$ and 
that at least one of the following
conditions is satisfied. 
\begin{enumerate}
\item $d(f)>1$;
\item $f$ is nonnegative or nonpositive on $U$;
\item $f_{\tau_*}$ 
does not vanish on $U\cap(\R\setminus\{0\})^n$.
\end{enumerate}
If 
$\Re(\varphi(0))>0$ 
$($resp. $\Re(\varphi(0))<0)$ 
and 
$\Re(\varphi)$ is nonnegative 
$($resp. nonpositive$)$ on $U$ 
and the support of $\varphi$ is
contained in a sufficiently small neighborhood 
of the origin, 
then we have 
\begin{equation}
C_{\pm}=\sum_{\theta\in\{-1,1\}^n}
\tilde{C}_{\pm}(f_{\theta},\varphi_{\theta})
\mbox{\,\,\, and \,\,\,} C=C_++C_-,
\label{eqn:9.28}
\end{equation}
where 
$f_{\theta}(x):=f(\theta x)$,
$\varphi_{\theta}(x):=\varphi(\theta x)$
and
$\tilde{C}_{\pm}(f,\varphi)$ 
$($defined as in $(\ref{eqn:9.18}))$
are as in  
$(\ref{eqn:9.22}),(\ref{eqn:9.24}),(\ref{eqn:9.25}),(\ref{eqn:9.26})$.
It follows from these formulae that   
$\Re(C_{\pm})$ are nonnegative and 
$\Re(C)=\Re(C_{+}+C_{-})$ is positive. 
\end{theorem}
\begin{proof}
From the relationship (\ref{eqn:9.3}), 
it suffices to show 
that the above conditions (ii) and (iii)
imply the condition (ii) in Proposition~9.6, 
when the support of $\varphi$ is contained in 
a sufficiently small neighborhood 
of the origin.

(ii)\quad 
Let us assume that $f$ is nonnegative on $U$. 
(Needless to say, the nonpositive case is similarly 
proved.)
Let $\sigma$ be in $\Sigma_*^{(n)}$. 
From the equation (\ref{eqn:8.7}),
$l(a^j(\sigma))$ are even
for all $j$ and $f_{\sigma}$ is nonnegative
on $\pi(\sigma)^{-1}(U)$. 
Let us assume that there exists a point 
$b_0\in T_I(\R^n)\cap \pi(\sigma)^{-1}(U)$ 
with nonempty $I\subset \{1,\ldots,n\}$ 
such that $f_{\sigma}(b_0)=0$. 
Since $f$ is nondegenerate with respect to
$\Gamma_+(f)$, 
Theorem~8.10 implies that there is a point 
$b \in \pi(\sigma)^{-1}(U)$ 
close to $b_0$ 
such that $f_{\sigma}(b)<0$.
This contradicts the nonnegativity 
of $f$ on $U$, 
so we see that  
there exists an open neighborhood $V$ such that 
$\{
y\in\pi(\sigma)^{-1}(V);f_{\sigma}(y)=0
\}\subset (\R\setminus\{0\})^n$
and $\Re(\varphi)$ is nonnegative on $V$.
By replacing $U$ by $V$, 
the condition (ii) in the above theorem 
implies the condition (ii) in Proposition~9.6.

(iii)\quad 
We only consider the case that
$f_{\tau_*}$ is positive on $U\cap (\R\setminus\{0\})^n$.
It follows from the equation (\ref{eqn:8.9}) that
$f_{\sigma}\circ T_{A(\sigma)}$ is 
nonnegative on $\pi(\sigma)^{-1}(U)$. 
By the same argument as in the above case (ii), 
the nondegeneracy condition implies
that $f_{\sigma}\circ T_{A(\sigma)}$ is
positive on $\pi(\sigma)^{-1}(U)$ with a sufficiently
small neighborhood $U$.
\end{proof}

\section{Proofs of the theorems in Section~3}

\subsection{Relationship between $I(t;\varphi)$ and 
$Z_{\pm}(s;\varphi)$}

It is known 
(see \cite{igu78}, \cite{agv88}, etc.) 
that 
the study of the asymptotic behavior of the 
oscillatory integral $I(t;\varphi)$ in (\ref{eqn:1.1})
can be reduced to an investigation of the poles
of the functions $Z_{\pm}(s;\varphi)$ in (\ref{eqn:9.1}). 
Here, we overview this situation. 
Let $f$,$\varphi$ satisfy 
the conditions (A),(B) in Section~3. 
Suppose that the support of $\varphi$ is sufficiently small. 

Define the Gelfand-Leray function  
$K:\R \to \R$ as
\begin{equation}
K(u)=\int_{W_u} \varphi(x)  \omega, 
\label{eqn:10.3}
\end{equation}
where $W_u=\{x\in\R^n; f(x)=u\}$ and $\omega$ is 
the surface element on $W_u$ which is determined by 
$
df\wedge \omega=dx_1\wedge\cdots\wedge dx_n.
$
By using $K(u)$, 
$I(t;\varphi)$ and $Z_{\pm}(s;\varphi)$ can be expressed 
as follows.
Changing the integral variables 
in (\ref{eqn:1.1}),(\ref{eqn:9.1}), we have 
\begin{equation}
\begin{split}
&I(t;\varphi)=\int_{-\infty}^{\infty}
e^{itu}K(u)du 
=\int_{0}^{\infty}
e^{i\tau t}K(u)du+
\int_{0}^{\infty}
e^{-itu }K(-u)du,
\end{split}
\label{eqn:10.2}
\end{equation}
\begin{equation}
Z_{\pm}(s;\varphi)
=\int_0^{\infty} u^s K(\pm u)du,
\label{eqn:10.3}
\end{equation}
respectively. 
Applying the inverse formula of the Mellin transform
to (\ref{eqn:10.3}),  
we have 
\begin{equation}
K(\pm u)
=\frac{1}{2\pi i}
\int_{c-i\infty}^{c+i\infty} 
Z_{\pm}(s;\varphi)u^{-s-1}ds,
\label{eqn:10.4}
\end{equation}
where $c>0$ and the integral contour follows
the line Re$(s)=c$ upwards. 
Let us consider the case that 
$Z_{\pm}(s;\varphi)$ 
are meromorphic functions on $\C$ and 
their poles exist on the negative part of the real axis. 
By deforming the integral contour as $c$ tends 
to $-\infty$ in (\ref{eqn:10.4}), 
the residue formula gives the
asymptotic expansions of $K(u)$ 
as $u\to\pm 0$.
Substituting these expansions of $K(u)$ into 
(\ref{eqn:10.2}), 
we can get an asymptotic expansion of $I(t;\varphi)$ 
as $t\to+\infty$.

Through the above calculation, 
we see a specific relationship for the coefficients. 
If $Z_{\pm}(s;\varphi)$  
have the Laurent expansions at $s=-\lambda$:
\begin{equation*}
Z_{\pm}(s;\varphi)=\frac{B_{\pm}}{(s+\lambda)^{\rho}}+
O\left(\frac{1}{(s+\lambda)^{\rho-1}}\right),
\end{equation*} 
then the corresponding part in the asymptotic
expansion of $I(t;\varphi)$ has the form
$$
B\tau^{-\lambda}(\log t)^{\rho-1}+
O(\tau^{-\lambda}(\log t)^{\rho-2}).
$$ 
Here a simple computation gives 
the following relationship:
\begin{equation}
B=\frac{\Gamma(\lambda)}{(\rho-1)!}
\left[
e^{i\pi \lambda/2}B_+ +e^{-i\pi \lambda/2}B_-
\right],
\label{eqn:10.5}
\end{equation}
where $\Gamma$ is the Gamma function. 
\subsection{Proofs of Theorems 3.3, 3.5 and 3.7}
Applying the above argument to the results 
relating to $Z_{\pm}(s;\varphi)$ in Section~9, 
we obtain the theorems in Section~3. 

{\it Proof of Theorems~3.3 and 3.5.}\quad
These theorems are shown by using Theorem~9.1 
with Proposition~9.5. 
Notice that Proposition~9.5 and 
the relationship (\ref{eqn:10.5}) 
induce the cancellation of the coefficients 
of the terms,  
whose orders are larger than $-1/d(f)$. 

{\it Proof of Theorem~3.7.}\quad 
This theorem follows from Theorem~9.10. 
Notice that the relationship (\ref{eqn:10.5}) 
gives the information about the
coefficient of the first term of $I(t;\varphi)$. 
\subsection{The first coefficient in the asymptotics (\ref{eqn:3.1})}

From the relationship (\ref{eqn:10.5}) and 
the equations (\ref{eqn:9.28}),
we give explicit formulae for the coefficient of 
the leading term of the asymptotic expansion in (\ref{eqn:3.1})  
as follows. 
\begin{theorem}
If $f$ satisfies the same conditions in Theorem~3.7 and 
the support of $\varphi$ is contained in a
sufficiently small neighborhood of the origin, 
then we have
\begin{equation*}
\begin{split}
&\lim_{t\to\infty}
t^{1/d(f)}(\log t)^{-m(f)+1}\cdot I(t;\varphi)\\
&\quad\quad=
\frac{\Gamma(1/d(f))}{(m(f)-1)!}
[
e^{i\pi/(2d(f))}C_+ +
e^{-i\pi/(2d(f))}C_-
],
\end{split}
\label{eqn:}
\end{equation*}
where $C_{\pm}$ are as in $(\ref{eqn:9.28})$.
\end{theorem}

\section{Examples}
In this section, we consider 
the oscillatory integrals (\ref{eqn:1.1})
with specific phases $f$, which satisfy
the condition (A) in Section~3 and
have noncompact principal faces. 
Moreover, in the first three examples, 
the phases belong to the class $\hat{\mathcal E}(U)$
and satisfy the nondegeneracy condition in 
Section~3.
These examples cannot be directly dealt with by earlier
investigations.
Note that, for each phase, the origin is not a critical point of 
finite multiplicity in Tougeron's theorem (see Remark~2.4).
The last example shows that  
Theorem~3.1 (iii) cannot be directly generalized 
to the smooth case.  
In this section, 
we assume that the amplitudes $\varphi$ satisfy the condition (B)
in Section~3. 

\subsection{Example 1}
Consider the following two-dimensional example:
\begin{equation}
f(x_1,x_2)
=x_1^8+x_1^{7} x_2^{1}+ x_1^{6} x_2^{2}(1+e^{-1/x_2^2}).
\end{equation}
It is easy to determine important quantities and 
functions as follows. 
Let $\sigma$ be a cone whose skeleton
is $a^1$, $a^2$, where $a^1=(1,0)$, $a^2=(1,1)$.
\begin{itemize}
\item $d(f)=6$ and $m(f)=1$,
\item $\tau_*=\{(6,\alpha_2);\alpha_2\geq 2\}$,
$\Sigma^{(2)}_*=\{\sigma\}$ and $A(\sigma)=\{1\}$,
\item $l(a^1)=6$, $l(a^2)=8$, 
\item $\pi(\sigma)(y_1,y_2)=(y_1y_2,y_2)$,
\item $f_{\tau_*}(x)=x_1^6 x_2^{2}(1+e^{-1/x_2^2})$ and  
$(f_{\tau_*}\circ\pi(\sigma))(y)=
y_1^{6} y_2^{8}(1+e^{-1/y_2^2})$,
\item $f_{\sigma}(y)=y_1^2+y_1+1+e^{-1/y_2^2}$
\end{itemize}
Substituting the above into (\ref{eqn:9.22}) or (\ref{eqn:9.24}),
we have 
\begin{equation*}
\tilde{C}_+(f,\varphi)=\frac{1}{6}
\int_{0}^{\infty}
\frac{\varphi(0,y)}{
y^{1/3}(1+e^{-1/y^2})^{1/6}
}dy 
\end{equation*}
and $\tilde{C}_-(f,\varphi)=0$.
Moreover, we have
\begin{equation*}
\lim_{t\to\infty}
t^{1/6}\cdot I(t;\varphi)
=
\frac{e^{\pi i/12}}{3}
\int_{-\infty}^{\infty}
\frac{\varphi(0,y)}
{|y|^{1/3}(1+e^{-1/y^2})^{1/6}
}dy. 
\end{equation*}
Note that strong results, relating to this example,
have been obtained 
in \cite{im11tams},\cite{im11jfaa} in the case where
the phase is smooth and 
the principal face is compact, and
in \cite{gre09} 
in the case where 
the principal face is noncompact but
the phase needs the real analyticity. 

\subsection{Example 2}
Consider the following three-dimensional example:
\begin{equation}
f(x_1,x_2,x_3)
=
x_1^6+ 
x_1^{4}x_2^2 e^{-1/x_3^2} +
x_1^{2}x_2^4 e^{-1/x_3^4} +
x_2^{6}.
\end{equation}
It is easy to determine important quantities and 
functions as follows. 
Let $\sigma$ be a cone whose skeleton
is $a^1$, $a^2$, $a^3$, 
where $a^1=(1,0,0)$, $a^2=(1,1,0)$, $a^3=(0,0,1)$.
\begin{itemize}
\item $d(f)=3$ and $m(f)=1$,
\item $\tau_*=\{\alpha\in\R_+^3;\alpha_1+\alpha_2=6\}$,
$\sigma\in\Sigma^{(3)}_*$ and $A(\sigma)=\{2\}$,
\item $l(a^1)=l(a^3)=0$ and $l(a^2)=6$, 
\item $\pi(\sigma)(y_1,y_2,y_3)=(y_1y_2,y_2,y_3)$,
\item $f_{\tau_*}(x)=f(x)$ and
 $(f_{\tau_*}\circ\pi(\sigma))(y)=
y_2^{6}(y_1^6+y_1^4 e^{-1/y_3^2} +
y_1^{2} e^{-1/y_3^4} +1)$,
\item $f_{\sigma}(y)=y_1^6+y_1^4 e^{-1/y_3^2} +
y_1^{2} e^{-1/y_3^4} +1$.
\end{itemize}
Substituting the above into (\ref{eqn:9.22}) or (\ref{eqn:9.24}),
we have 
\begin{equation*}
\tilde{C}_+(f,\varphi)
=\frac{1}{6}
\int_{\R_+^2}
\frac{
\varphi(y_1,0,y_3)
}{
(y_1^6+y_1^4 e^{-1/y_3^2} +
y_1^{2} e^{-1/y_3^4} +1)^{1/3}
}dy_1 dy_3
\end{equation*} 
and 
$\tilde{C}_-(f,\varphi)=0$.  Moreover, 
we have
\begin{equation*}
\lim_{t\to\infty}
t^{1/3}\cdot I(t;\varphi)=
\Gamma(4/3)e^{\pi i/6}
\int_{\R^2}
\frac{
\varphi(y_1,0,y_3)}{
(y_1^6+y_1^4 e^{-1/y_3^2} +
y_1^{2} e^{-1/y_3^4} +1)^{1/3}
}dy_1 dy_3.
\end{equation*}

\subsection{Example 3}
In the case of 
the following three-dimensional example,
the logarithmic factor appears 
in the leading term of asymptotics of
$I(t;\varphi)$:
\begin{equation}
f(x_1,x_2,x_3)
=
x_1^6+ 
x_1^{2}x_2^2(1+e^{-1/x_3^2})+ 
x_2^{6}.
\end{equation}
It is easy to determine important quantities and 
functions as follows. 
Let $\sigma$ be a cone whose skeleton
is $a^1$, $a^2$, $a^3$, 
where $a^1=(2,1,0)$, $a^2=(1,1,0)$, $a^3=(0,0,1)$.
\begin{itemize}
\item $d(f)=2$ and $m(f)=2$,
\item $\tau_*=\{(2,2,\alpha_3);\alpha_3\geq 0\}$,
$\sigma\in\Sigma^{(3)}_*$ and $A(\sigma)=\{1,2\}$,
\item $l(a^1)=6$, $l(a^2)=4$, $l(a^3)=0$, 
\item $\pi(\sigma)(y_1,y_2,y_3)=(y_1^2y_2,y_1y_2,y_3)$,
\item $f_{\tau_*}(x)=x_1^2x_2^2(1+e^{-1/x_3^2})$ and
 $(f_{\tau_*}\circ\pi(\sigma))(y)=
y_1^{6}y_2^4(1+e^{-1/y_3^2})$,
\item $f_{\sigma}(y)=y_1^6 y_2^2+y_2^2+1+e^{-1/y_3^2}$. 
\end{itemize}
Substituting the above into (\ref{eqn:9.22}) or (\ref{eqn:9.24}),
we have 
\begin{equation}
\tilde{C}_+(f,\varphi)
=\frac{1}{24}
\int_{0}^{\infty}
\frac{
\varphi(0,0,y)}{
(1+e^{-1/y^2})^{1/2} 
}dy
\end{equation}
and $\tilde{C}_-(f,\varphi)=0$. Moreover, we have
\begin{equation*}
\lim_{t\to\infty}
t^{1/2}(\log t)^{-1}\cdot I(t;\varphi)=
\frac{\sqrt{\pi}e^{i\pi/4}}{6}
\int_{-\infty}^{\infty}
\frac{
\varphi(0,0,y)
}{
(1+e^{-1/y^2})^{1/2}}dy.
\end{equation*}
\subsection{Example 4}
Consider the following two-dimensional example
given by 
Iosevich and Sawyer in \cite{is97}:
\begin{equation}
f(x_1,x_2)
=
x_1^2+ 
e^{-1/|x_2|^{\alpha}}, \quad\quad \alpha>0.
\end{equation}
Note that the above $f$ 
satisfies the nondegeneracy condition as in Section~3 but
it does not belong to $\hat{\mathcal E}(U)$. 
It is easy to see the following:
\begin{itemize}
\item $d(f)=2$ and $m(f)=1$,
\item $\tau_*=\{\a\in\R_+^2;\alpha_1=2\}$,
\item $f_{\tau_*}(x_1,x_2)=x_1^2$.
\end{itemize}
Consider an amplitude of the form:   
$\varphi(x_1,x_2)=\psi_1(x_1)\psi_2(x_2)$ where
$\psi_j$ are smooth nonnegative functions on $\R$ 
satisfying $\psi_j(0)>0$ and 
its support is small for $j=1,2$.
In \cite{is97}, 
Iosevich and Sawyer shows: 
$$
|I(t;\varphi)|\leq Ct^{-1/2}(\log t)^{-1/\alpha} 
\,\, \mbox{ for $t\geq 2$.}
$$
In particular, we have $\lim_{t\to\infty} t^{1/2}I(t;\varphi)=0$, 
which is different phenomenon from 
that in Remark~3.2. 
This example shows that 
the assertion (iii) in Theorem~3.1 with Remark~3.2 
cannot be directly generalized to the smooth case. 
Moreover, 
the pattern of the asymptotic expansion in this case 
might be different from 
that of (\ref{eqn:3.1}). 


\vspace{1 em}

{\sc Acknowledgements.}\quad 
The authors would like to express their sincere gratitude 
to Hiroyuki Ochiai for his careful reading of the manuscript
and giving the authors many valuable comments.




\end{document}